\documentclass[letterpaper,11pt, reqno]{article}

\usepackage{algpseudocode}
\usepackage{algorithm}
\usepackage{amsfonts}
\usepackage[leqno]{amsmath}
\usepackage{amssymb}
\usepackage{amstext}
\usepackage{amsthm}
\usepackage[toc,page]{appendix}
\usepackage{authblk}
\usepackage{booktabs}
\usepackage[font=small]{caption}
\usepackage{cite}
\usepackage{datetime}
\usepackage{enumitem}
\usepackage{etoolbox}
\usepackage[T1]{fontenc}
\usepackage{framed}
\usepackage{fullpage}
\usepackage[margin=1in]{geometry}
\usepackage{graphics}
\usepackage{graphicx}
\usepackage{hyperref}
\usepackage[utf8]{inputenc}
\usepackage{mathrsfs}
\usepackage[framemethod=TikZ]{mdframed}
\usepackage{microtype}
\usepackage{multirow}
\usepackage{nicefrac}
\usepackage{pdflscape}
\usepackage{pgfplots}
\usepackage{rotating}
\usepackage{setspace}
\usepackage[position=b]{subcaption}
\usepackage{tikz}
\usepackage{times}
\usepackage{units}
\usepackage{url}
\usepackage{xcolor}

\usepackage{bbm}

\allowdisplaybreaks

\newcommand{\res}{%
  \,\raisebox{-.127ex}{\reflectbox{\rotatebox[origin=br]{-90}{$\lnot$}}}\,%
}

\pgfplotsset{compat=newest}

\definecolor{darkred}{rgb}{0.6,0.1,0.1}
\definecolor{darkgreen}{rgb}{0.1,0.6,0.1}
\definecolor{darkblue}{rgb}{0.1,0.1,0.6}

\newmdenv[
roundcorner=10pt,
backgroundcolor=gray!10,
linecolor=gray!10,
tikzsetting={draw=black,line width=3pt,dashed,dash pattern= on 10pt off 3pt},
outerlinewidth=0pt,
innerlinewidth=0pt,
]{figbox}

\newlength\figureheight
\newlength\figurewidth

\newcommand{\iid}{\stackrel{\mathrm{iid}}{\sim}}
\DeclareMathOperator*{\argmax}{\mathrm{argmax}}

\def\dd{\mathrm{d}}
\newcommand{\dist}{\mathrm{dist}}

\def\eps{\varepsilon}
\DeclareMathOperator*{\essinf}{\mathrm{ess}\,\mathrm{inf}}
\DeclareMathOperator*{\esssup}{\mathrm{ess}\,\mathrm{sup}}

\newcommand{\Id}{\mathrm{Id}}
\newcommand{\Lip}{\mathrm{Lip}}

\newcommand{\Vol}{\mathrm{Vol}}
\DeclareMathOperator*{\Glim}{\Gamma\text{-}\lim}
\DeclareMathOperator*{\Gammato}{\stackrel{\Gamma-(\mathrm{d})}{\longrightarrow}}

\DeclareMathOperator*{\weakstarto}{\stackrel{{\mathit{w}}^*}{\rightharpoonup}} 
\DeclareMathOperator*{\TLpto}{\stackrel{\mathit{TL^p}}{\longrightarrow}}
\def\l{\left(}
\def\r{\right)}
\def\la{\left|}
\def\ra{\right|}
\def\lb{\left\{}
\def\rb{\right\}}
\def\rd{\right.}
\def\ls{\left[}
\def\rs{\right]}

\newcommand{\bbN}{\mathbb{N}}

\newcommand{\bbS}{\mathbb{S}}
\newcommand{\bbR}{\mathbb{R}}
\newcommand{\cA}{\mathcal{A}}
\newcommand{\cC}{\mathcal{C}}

\newcommand{\cE}{\mathcal{E}}
\newcommand{\cF}{\mathcal{F}}
\newcommand{\cG}{\mathcal{G}}
\newcommand{\cH}{\mathcal{H}}
\newcommand{\cI}{\mathcal{I}}

\newcommand{\cK}{\mathcal{K}}
\newcommand{\cL}{\mathcal{L}}

\newcommand{\cP}{\mathcal{P}}

\newcommand{\cU}{\mathcal{U}}

\newtheorem{theorem}{Theorem}[section]
\newtheorem{lemma}[theorem]{Lemma}
\newtheorem{proposition}[theorem]{Proposition}
\newtheorem{corollary}[theorem]{Corollary}
\newtheorem{mydef}[theorem]{Definition}
\theoremstyle{remark}
\newtheorem{remark}[theorem]{Remark}

\makeatletter
\newcommand{\leqnomode}{\tagsleft@true}
\newcommand{\reqnomode}{\tagsleft@false}

\newcounter{num}
\usepackage{enumitem, hyperref}
\makeatletter
\def\namedlabel#1#2{\begingroup
    #2%
    \def\@currentlabel{#2}%
    \phantomsection\label{#1}\endgroup
}
\makeatother

\title{Large Data Limit for a Phase Transition Model with the $p$-Laplacian on Point Clouds}
\author[1]{Riccardo Cristoferi}
\author[2]{Matthew Thorpe}
\affil[1]{Department of Mathematical Sciences,\protect\\ Carnegie Mellon University,\protect\\ Pittsburgh, PA 15213, USA \vspace{\baselineskip}}
\affil[2]{Department of Applied Mathematics and Theoretical Physics,\protect\\ University of Cambridge,\protect\\ Cambridge, CB3 0WA, UK}
\date{July 2018}

\begin{document}

\maketitle
\newcounter{broj}

\begin{abstract}
The consistency of a nonlocal anisotropic Ginzburg-Landau type functional for data classification and clustering is studied.
The Ginzburg-Landau objective functional combines a double well potential, that favours indicator valued functions, and the $p$-Laplacian, that enforces regularity.
Under appropriate scaling between the two terms minimisers exhibit a phase transition on the order of $\eps=\eps_n$ where $n$ is the number of data points.
We study the large data asymptotics, i.e. as $n\to \infty$, in the regime where $\eps_n\to 0$.
The mathematical tool used to address this question is $\Gamma$-convergence.
It is proved that the discrete model converges to a weighted anisotropic perimeter.
\end{abstract}

\section{Introduction \label{sec:Intro}}

The analysis of big data is one of the most important challenges we currently face.
A typical problem concerns partitioning the data based on some notion of similarity.
When the method makes use of a (usually small) subset of the data for which there are labels then this is known as a \emph{classification problem}.
When the method only uses geometric features, i.e. there are are no a-priori known labels, then this is known as a \emph{clustering problem}.
We refer to both problems as labelling problems.

A popular method to represent the geometry of a given data set is to construct a graph embedded in an ambient space $\bbR^d$.
Typically the labelling task is fulfilled via a minimization procedure.
In the machine learning community, successfully implemented approaches include minimizing graph cuts and total variation (see, for instance, \cite{arona, boykov, breslau, breslau2, breslau3, bres, ekaber, shi00, szlambress, szlambress2, slepcev17}).

Of capital importance for evaluating a labelling method is whether it is \emph{consistent} or not; namely it is desirable that the minimization procedure approaches some limit minimization method when the number of elements of the data set goes to infinity.
Indeed, knowing whether a specific minimization strategy is an approximation of a limit (minimizing) object can help explain properties of the finite data algorithm.
For a consistent methodology properties of the large data limit will be evident when a large, but finite, number of data points is being considered.
In particular, this can also be used to justify, a posteriori, the use of a certain procedure in order to obtain some desired features of the classification.
Furthermore, understanding the large data limits can open up new algorithms. \vspace{0.5\baselineskip}

This paper is part of an ongoing project aimed at justifying analytically the consistency of several models for soft labelling used by practitioners.
Here we consider a generalization of the approach introduced by Bertozzi and Flenner in \cite{bertozzi12} (see also \cite{calatroni17}, for an introduction on this topic see \cite{vanGenCarola}), where a Ginzburg-Landau (or Modica-Mortola, see \cite{modica87,modica77}) type functional is used as the underlining energy to minimize in the context of the soft classification problem.
The functional we consider is a discretisation of the non-local Ginzburg-Landau functional studied by Alberti and Bellettini~\cite{alberti98,alberti98a} with the generalisation that we consider 
non-uniform densities and $\ell^p$ (rather than $\ell^2$) cost on finite differences.
Our goal is to prove the consistency of the model.

There are multiple extensions to the approach we consider here; for instance our Ginzburg-Landau functional is based on the $p$-Laplacian, one can also consider the normalised $p$-Laplacian or the random walk Laplacian~(see~\cite{shi00, ng01}). 
Further open problems concern the extention to multi-phase labelling (see \cite{BertozziFlenner}) and convergence of the associated gradient flows. \vspace{0.5\baselineskip}

The paper is organized as follows: in the following subsection we define the discrete model, and in Subsection~\ref{subsec:Intro:InfinModel} we define the continuum limiting problem.
The main results are given in Section~\ref{sec:MainRes} with the proofs presented in Sections~\ref{sec:ConvGraph} and~\ref{sec:Const}.
In Section~\ref{subsec:Intro:LitRev} we give an overview on the related literature.
In Sections~\ref{subsec:Intro:pEx} and~\ref{subsec:Intro:AnisoEx} we include two examples with the purpose of demonstrating key properties of our functional; in particular how the choice of $p$ effects minimizers of our Ginzburg-Landau functional and an example to
image segmentation.
Section~\ref{sec:Back} contains some prelimimary material we include for the convenience of the reader.
Finally, Section~\ref{sec:ConvNLContinuum} is devoted to the proofs of some technical results that are of interest in their own right, and are later used in the proofs in Section~\ref{sec:ConvGraph}.

\subsection{Finite Data Model \label{subsec:Intro:FinModel}}

In the graph representation of a data set, vertexes are points $X_n:=\{x_i\}_{i=1}^n\subset X$, where $X\subset \bbR^d$ is a connected, bounded and open set, with weighted edges $\{W_{ij}\}_{i,j=1}^n$, where each $W_{ij}\geq0$ is meant to represent similarities between the vertexes $x_i$ and $x_j$, and in some sense encode the geometry.
The larger $W_{ij}$ is the more similar the points $x_i$ and $x_j$ are and "the closer they are on the graph".

Let us consider the problem of partitioning a set of data into two classes.
A partition of the set of points $X_n$ is a map $u:X_n\to\{-1,1\}$, where $-1$ and $1$ represent the two classes. 
This is referred to as \emph{hard} labelling, since $u$ can only assume a finite number of values.
From the computational point of view it is preferable to work with functions whose values range in the whole interval $[-1,1]$, \emph{i.e.}, labellings $u:X_n\to[-1,1]$, thus allowing for a \emph{soft} labelling.
Labels that are close to $1$, or to $-1$, are supposed to be in the same class. 
The model used to obtain the binary classification should then force the labelling to be either $1$ or $-1$ when the number of data points is large.

In order to scale the weights on the edges of the graph we define $W_{ij}$ through a kernel $\eta:\bbR^d\to\bbR$.
More precisely, we define the graph weights by $W_{ij} = \eta_\eps(x_i-x_j) = \frac{1}{\eps^d} \eta((x_i-x_j)/\eps)$ where $\eps$ controls the scale of interactions on the graph; in particular choosing $\eps$ large implies the graph is dense, and choosing $\eps$ small implies the graph is disconnected.
Later assumptions, see Remark~\ref{rem:MainRes:Connected}, imply that we scale $\eps=\eps_n$ such that the graph is eventually connected (with probability one).

We now introduce the discrete functional we are going to study.

\begin{mydef}
For $p\geq 1$ and $n\in\mathbb{N}$ define the functional $\cG_n^{(p)}:L^1(X_n) \to [0,\infty)$ by
\[
\cG_n^{(p)}(u) := \frac{1}{\eps_n n^2} \sum_{i,j=1}^n W_{ij} |u(x_i) - u(x_j)|^p + \frac{1}{\eps_n n} \sum_{i=1}^n V(u(x_i))\,,
\]
where
\begin{equation}\label{eq:wij}
W_{ij} := \eta_{\eps_n}(x_i-x_j) := \frac{1}{\eps_n^d}\eta\l\frac{x_i-x_j}{\eps_n}\r
\end{equation}
and $V$ is a double well potential (see Assumption~(B4)).
\vspace{0\baselineskip}
\end{mydef}

The first term in $\cG_n$ plays the role of penalising oscillations, intuitively one wants a labelling solution such that if $x_i$ and $x_j$ are close on the graph then the labels are also close.
This term, when $p=2$, can also be written as $\frac{1}{\eps_n n} \langle u,L u\rangle_{\mu_n}$ where $L$ is the graph Laplacian.
The second term penalises soft labellings.
In particular, we assume that $V(t)=0$ if and only if $t\in \{\pm 1\}$ and $V(t)>0$ for all $t\neq \pm 1$.
Hence any soft labelling is given a penalty of $\frac{1}{\eps_n n}\sum_{i=1}^n V(u(x_i))$, as $\eps_n\to 0$ this penality blows up unless $u$ takes the values $\pm 1$ almost everywhere.

The function $\eta$ plays the role of a mollifier, and that explains the definition of $\eta_{\eps_n}$.
Moreover, to justify the scaling $\frac{1}{\eps_n}$ we reason as follows: assume $\eta$ has support contained in a ball, we get
\[
|u(x_i) - u(x_j)|^p\sim \eps_n^{p}|\nabla u|^p\,.
\]
So that, dividing by $\eps_n$ will give us the typical form of the singular perturbation used in the gradient theory of phase transitions (see \cite{modica87}), namely
\[
\int_X \frac{1}{\eps_n}V(u)+\eps_n^{p-1}|\nabla u|^p\,.
\]

The consistency of the model is studied using $\Gamma$-convergence (see Section \ref{sec:Gammaconv}), a very important tool introduced by De Giorgi in the 70's to understand the limiting behavior of a sequence of functionals (see \cite{DG}).
This kind of variational convergence gives, almost immediately, convergence of minimizers.


\subsection{Infinite Data Model \label{subsec:Intro:InfinModel}}

In order to define the limiting functional, we first introduce some notation.

\begin{mydef}
Let $\nu\in\bbR^d$. Define $\nu^\perp:=\{ z\in \bbR^d \, : \, z\cdot \nu = 0 \}$.
Moreover, for $x\in\bbR^d$, set
\[
\cC(x,\nu) := \lb C\subset\nu^\perp \,:\, C \text{ is a } (d-1)\text{-dimensional cube centred at } x \rb \,.
\]
For $C\in\cC(x,\nu)$, we denote by $v_1,\dots,v_{d-1}$ its principal directions (where each $v_i$ is a unit vector normal to the $i^\text{th}$ face of $C$),
and we say that a function $u:\bbR^d\to\bbR$ is \emph{$C$-periodic} if $u(y+rv_i) = u(y)$ for all $y\in\bbR^d$, all $r\in\mathbb{N}$ and all $i=1,\dots,d-1$.

Finally, we consider the following space of functions:
\[
\cU(C,\nu) := \lb u:\bbR^d\to [-1,1] \, : \, u \text{ is } C\text{-periodic,} \lim_{y\cdot\nu\to \infty} u(y) = 1, \text{ and } \lim_{y\cdot\nu\to -\infty} u(y) = -1 \rb \,.
\]
\vspace{0\baselineskip}
\end{mydef}

We now define the limiting (continuum) model.

\begin{mydef}
Let $p\geq1$ and $X\subset \bbR^d$ be open and bounded, and $\rho\in L^\infty(X)$ a positive function.
Define the functional $\cG_\infty^{(p)}:L^1(X)\to [0,\infty]$ by
\[
\cG_\infty^{(p)}(u):=
\left\{
\begin{array}{ll}
\displaystyle\int_{\partial^*\{u=1\}} \sigma^{(p)}(x,\nu_u(x)) \rho(x) \, \dd \cH^{d-1}(x) & \text{if } u\in BV(X;\{\pm 1\})\,, \\
&\\
+\infty & \text{else}\,,
\end{array}
\right.
\]
where
\[
\sigma^{(p)}(x,\nu) := \inf\lb \frac{1}{\cH^{d-1}(C)} G^{(p)}(u,\rho(x),T_C) \, : \, C\in \cC(x,\nu)\,, u\in \cU(C,\nu) \rb\,,
\]
and, for $C\in \cC(x,\nu)$, we set $T_C:= \lb z+t\nu \, : \, z\in C, t\in \bbR \rb$.
Finally, for $\lambda\in\bbR$ and $A\subset\bbR^d$ define
\[
G^{(p)}(u,\lambda,A):= \lambda \int_{A} \int_{\bbR^d} \eta(h) |u(z+h) - u(z)|^p \, \dd h \, \dd z + \int_{A} V(u(z)) \, \dd z\,.
\]
Here $\partial^*\{u=1\}$ denotes the reduced boundary of $\{u=1\}$ and $\nu_u(x)$ is the measure theoritic exterior normal to the set $\{u=1\}$ at the point $x\in \partial^*\{u=1\}$ (see Definition \ref{eq:defnormal}).
\vspace{0\baselineskip}
\end{mydef}

Notice that the discrete functional $\cG_n^{(p)}$ is nonlocal while the functional $\cG_\infty^{(p)}$ is local.
The minimization problem defining $\sigma^{(p)}$ is called the \emph{cell problem} and it is common in phase transitions problems (see related works in Section~\ref{subsec:Intro:LitRev}).
Although not explicit, we have at least information on the form of the limiting functional: an anisotropic weighted perimeter.
This shows that minimizers of $\cG_\infty^{(p)}$ are sets $E\subset X$ whose boundary $\partial E$ (or, to be precise, the reduced boundary $\partial^* E$) will likely be in the region where $\rho$ is small and orthogonal to directions $\nu$ for which $\sigma^{(p)}(x,\nu)$ is small. \vspace{0.5\baselineskip}

Finally, we want to point out that one of the main issues we have to deal with is that, for each $n\in\bbN$, the data set $X_n$ is a discrete set, while in the limit the data is given by a probability measure $\mu$ on the set $X$, hence why we call $\cG_\infty^{(p)}$ the continuum model.
Thus, we will need to compare functions (the labeling) defined on different sets. To do so we will implement the strategy introduced by Garc\'{i}a Trillos and Slep\v{c}ev in \cite{garciatrillos16}, that consists in extending a function $u:X_n\to\bbR$ to a function $v:X\to\bbR$ in an optimal piecewise constant way. Optimal here is meant in the sense of optimal transportation. 
In particular, a sequence of functions $\{u_n\}_{n=1}^\infty$ with $u_n\in L^1(X_n)$, is said to converge in the $TL^1$ topology to a function $u\in L^1(X)$ if there exists a sequence $\{T_n\}_{n=1}^\infty\subset L^1(X;X_n)$ converging to the identity map in $L^1(X)$ and with
\[ \mu(T^{-1}_n(B)) = \frac{1}{n} \# \lb x_i\in B \, : \, i=1,2,\dots, n\rb \]
for every Borel set $B\subset X$, such that $u_n\circ T_n\to u$ in $L^1(X)$.
We review the $TL^1$ topology in more detail in Section~\ref{subsec:Prelim:Trans}.


\subsection{Main Results \label{sec:MainRes}}

This section is devoted to the precise statements of the main results of this paper.

Let $X\subset \bbR^d$ be a bounded, connected and open set with Lipschitz boundary. 
Fix $\mu\in \cP(X)$ and assume the following.
\begin{itemize}
\item[(A1)] $\mu \ll \cL^d$, has a continuous density $\rho:X\rightarrow[c_1,c_2]$ for some $0<c_1\leq c_2<\infty$.
\end{itemize}

We extend $\rho$ to a function defined in the whole space $\bbR^d$ by setting $\rho(x):=0$ for $x\in\bbR^d\setminus X$.
For all $n\in\mathbb{N}$, consider a point cloud $X_n=\{x_i\}_{i=1}^n\subset X$ and let $\mu_n$ be the associated empirical measure (see Definition \ref{def:empmeas}).
Let $\{\eps_n\}_{n=1}^\infty$ be a positive sequence converging to zero and such that the following rate of convergence holds:
\begin{itemize}
\item[(A2)] $\displaystyle \frac{\dist_\infty(\mu_n,\mu)}{\eps_n} \to 0\,,$ where $\dist_\infty(\mu_n,\mu)$ is the $\infty$-Wasserstein distance between the measures $\mu_n$ and $\mu$, see Definition~\ref{def:Back:Trans:Wass}.
\end{itemize}

\begin{remark}
\label{rem:MainRes:Connected}
When $x_i\iid \mu$ then (with probability one), Assumption~(A2) is implied by $\eps_n \gg \delta_n$,
where $\delta_n$ is defined in Theorem \ref{thm:optrate}.
Notice that for $d\geq 3$ this lower bound on $\eps_n$ ensures that the graph with vertices $x_n$ and edges weighted by $W_{ij}$ (see \eqref{eq:wij}) is eventually connected (see \cite[Theorem 13.2]{penrose}).
The lower bound can potentially be improved when $x_i$ are not independent.
For example if $\{x_i\}_{i=1}^n$ form a regular graph then $\mu_n$ converges to the uniform measure and the lower bound is given by
$\eps_n \gg n^{-\frac{1}{d}}$. \vspace{0.8\baselineskip}
\end{remark}

The double well potential $V:\bbR\to\bbR$ satisfies the following.

\begin{itemize}
\item[(B1)] $V$ is continuous. 
\item[(B2)] $V^{-1}\{0\} = \{\pm 1\}$ and $V\geq0$.
\item[(B3)] There exists $\tau>0, R_V>1$ such that for all $|s|\geq R_V$ that $V(s)\geq \tau|s|$.
\item[(B4)] $V$ is Lipschitz continuous on $[-1,1]$.
\end{itemize}

The assumptions on $V$ imply that in the limit there are only two phases $\pm 1$.
Assumption (B3) is used to establish compactness, in particular it is used to show that minimisers can be bounded in $L^\infty$ by~1. The prototypical example of a function $V:\bbR^d\to\bbR$ satisfying (B1-4) is given by $V(s):=(s^2-1)^2$.

Recall that the graph weights are defined by $W_{ij} = \eta_{\eps_n}(x_i-x_j)$.
We assume that $\eta:\bbR^d\to[0,\infty)$ is a measurable function satisfying the following.

\begin{itemize}
\item[(C1)] $\eta\geq0$, $\eta(0)>0$ and $\eta$ is continuous at $x=0$.
\item[(C2)] $\eta$ is an even function, i.e. $\eta(-x) = \eta(x)$.
\item[(C3)] $\eta$ has support in $B(0,R_\eta)$, for some $R_\eta>0$.
\item[(C4)] For all $\delta>0$ there exists $c_\delta,\alpha_\delta$ such that if $|x-z|\leq \delta$ then $\eta(x) \geq c_\delta \eta(\alpha_\delta z)$, furthermore $c_\delta\to 1$, $\alpha_\delta\to 1$ as $\delta\to 0$.
\end{itemize}

\begin{remark}\label{rem:eta}
Note that (C3) and (C4) imply that $\|\eta\|_{L^\infty}<\infty$ and, in particular, $\int_{\bbR^d} \eta(x) |x| \, \dd x < \infty$.
Indeed, given $\delta>0$, it is possible to cover $B(0,R_\eta)$ with a finite family
$\widetilde{B}_\delta(x_1),\dots,\widetilde{B}_\delta(x_r)$ of sets of the form
\[
\widetilde{B}_\delta(x_i):=\{\, \alpha_\delta z \,:\, |z-x_i|<\delta \,\}\,.
\]
\vspace{0\baselineskip}
\end{remark}

Assumption~(C2) is justified by the fact that $\eta$ plays the role of an interaction potential.
Finally, Assumption~(C4) is a version of continuity of $\eta$ we need in order to perform our technical computations.
We note that (C4) is general enough to include $\eta(x) = \chi_{A}$ where $A\subset \bbR^d$ is open, bounded, convex and $0\in A$, see~\cite[Proposition 2.2]{thorpe17AAA}. 

The main result of the paper is the following theorem.

\begin{theorem}
\label{thm:MainRes:Compact&Gamma}
Let $p\geq 1$ and assume (A1-2), (B1-4) and (C1-4) are in force.
Then, the following holds:
\begin{itemize}
\item (compactness) let $u_n\in L^1(\mu_n)$ satisfy $\sup_{n\in \bbN} \cG_n^{(p)}(u_n) < \infty$, then $u_n$ is relatively compact in $TL^1$ and each cluster point $u$ has $\cG_\infty^{(p)}(u)<\infty$;
\item ($\Gamma$-convergence) $\Glim_{n\to \infty}(TL^1) \cG_n^{(p)} = \cG_\infty^{(p)}$.
\end{itemize}
\vspace{0\baselineskip}
\end{theorem}

The proof of compactness is in Section~\ref{subsec:ConvGraph:Compact} and the $\Gamma$-convergence is proved in Sections~\ref{subsec:ConvGraph:Liminf} and~\ref{subsec:ConvGraph:Limsup}.

Since the proof of Theorem \ref{thm:MainRes:Compact&Gamma} is quite long, we briefly sketch here the main idea behind the $\Gamma$-convergence result.
We approximately follow the method of~\cite{garciatrillos16} where the authors considered the continuum limit of total variation on point clouds.
We will show the convergence of the discrete nonlocal functional $\cG^{(p)}_n$ to the continuum local one $\cG_\infty^{(p)}$ via an intermediate nonlocal continuum functional $\cF^{(p)}_{\eps_n}$ (defined in~\eqref{eq:ConvNLContinuum:Feps}).
In particular, we will prove that:
\begin{itemize}
\item[(i)] the functionals $\cF^{(p)}_{\eps_n}$ $\Gamma$-converge in $L^1(X)$ to $\cG_\infty^{(p)}$, see Section~\ref{sec:ConvNLContinuum}, where we implement a strategy similar to the one of \cite{alberti98}, where the authors considered the functional $\cF^{(p)}_{\eps_n}$ with $\rho\equiv1$ and $p=2$,
\item[(ii)] it is possible to bound from below $\cG^{(p)}_n$ with $\cF^{(p)}_{\eps'_n}$ (see \eqref{eq:ineq2}), where $\lim_{n\to \infty}\frac{\eps^\prime_n}{\eps_n}=1$, from which the liminf inequality follows,
\item[(iii)] if $u\in BV(X;\{\pm 1\})$ we use the same recovery sequence $u_{\eps}$ as in~\cite{alberti98} to show $\limsup_{\eps\to 0}\cF_{\eps}^{(p)}(u_\eps) \leq \cG_\infty^{(p)}$, after which we can get an upper bound of $\cG^{(p)}_n(u_n)$ in terms of $\cF^{(p)}_{\eps'_n}(u_{\eps'_n})$ where $u_n:X_n\to[0,\infty)$ is defined at each $x_i\in X_n$ as a suitable average of $u_{\eps'_n}$ around the point $x_i$ and $\lim_{n\to\infty}\frac{\eps'_n}{\eps_n}=1$. This will give us the limsup inequality.
\end{itemize}
Similarly, the compactness property follows by comparing $\cG_n^{(p)}$ with the intermediary functional $\cF_{\eps_n}^{(p)}$.

As an application of the Theorem \ref{thm:MainRes:Compact&Gamma}, we consider the functional $\cG^{(p)}_n$ with a data fidelity term.

\begin{mydef}\label{def:fidelityterm}
Let $k_n:X_n\times \bbR \to\bbR$ and $k_\infty:X\times \bbR\to\bbR$.
Define the functionals $\cK_n: TL^1(X_n)\to\bbR$ and $\cK_\infty: TL^1(X)\to\bbR$ by
\[
\cK_n(u,\nu) = \lb \begin{array}{ll} \frac{1}{n} \sum_{i=1}^n k_n(x_i,u(x_i)) & \text{if } \nu=\mu_n, \\ +\infty & \text{else,} \end{array} \rd
\]
and
\[
\cK_\infty(u,\nu) = \lb \begin{array}{ll} \int_X k_\infty(x,u(x)) \rho(x) \, \dd x & \text{if } \nu=\mu, \\ +\infty & \text{else} \end{array} \rd
\]
respectively. \vspace{0.5\baselineskip}
\end{mydef}

We make the following assumptions on $k_n, k_\infty$:
\begin{itemize}
\item[(D1)] $k_n\geq 0$, $k_\infty\geq0$.
\item[(D2)] There exist $\beta>0$ and $q\geq 1$ such that $k_n(x,u) \leq \beta(1+|u|^q)$, for all $n\in\mathbb{N}$ and almost all $x\in X_n$.
\item[(D3)] For almost every $x\in X$ the following holds: let $u_n\to u$ be a converging real valued sequence and $x_n\to x$, then
\[
\lim_{n\to \infty} k_n \l x_n,u_n \r = k_\infty(x,u)\,.
\]
\end{itemize}

\begin{remark}
\label{rem:Intro:MainRes:SLL}
For example, we can use this form of $\cK_n,\cK_\infty$ to include a data fidelity term in a specific subset of $X$.
Let $B\subset X$ be an open set with $\Vol(B)>0$ and $\Vol(\partial B) = 0$.
Let $\lambda_n\geq0$ with $\lambda_n\to\lambda$ as $n\to\infty$.
Let $y_n\in L^1(X_n)$ and $y_\infty\in L^1(X)$ with $\sup_{n\in \bbN}\|y_n\|_{L^\infty}<\infty$ and such that $y_n(x_{i_n})\to y_\infty(x)$ for almost every $x\in X$ and any sequence $x_{i_n}\to x$.
Define
\[
k_n(x,u) :=
\left\{
\begin{array}{ll}
\lambda_n |y_n(x) - u|^q & \text{ in } B\cap X_n\,, \\
0 & \text{ on } X_n\setminus B\,,
\end{array}
\right.
\]
\[
k_\infty(x,u):=
\left\{
\begin{array}{ll}
\lambda |y_\infty(x) - u|^q & \text{ in } B\,, \\
0 & \text{ on } X\setminus B\,.
\end{array}
\right.
\]
Then $k_n$ and $k_\infty$ satisfy Assumptions~(D1-3). Indeed, (D1) follows directly from the definition of the fidelity terms, while (D3) holds thanks to continuity and the fact that $\Vol(\partial B) = 0$. Finally, in order to prove (D2) we simply notice that
\[
k_n(x,u) = \lambda_n| y_n(x) - u|^q \leq \sup_{n\in \bbN} \lambda_n 2^{q-1} \l \|y_n\|^q_{L^\infty} + |u|^q \r\leq \beta (1+|u|^q)\,.
\]
for some $\beta>0$. \vspace{0.8\baselineskip}
\end{remark}

We now consider the minimisation problem
\[
\text{minimise} \quad \cG_n^{(p)}(u) + \cK_n(u) \quad \text{over } u \in L^1(X_n)\,.
\]

\begin{corollary}
\label{cor:MainRes:Constrained}
In addition to Assumptions (A1-3), (B1-2), (C1-4), (D1-3), assume that for the same $q\geq 1$ as in Assumption (D2) there exists $\tau,R_V>0$ such that for all $|s|\geq R_V$ that $V(s) \geq \tau|s|^q$.
Then any sequence of almost minimizers of $\cG_n^{(p)}+\cK_n$ is compact in $TL^1$.
And furthermore, any cluster point of almost minimizers is a minimizer of $\cG_\infty^{(p)}+\cK_\infty$ in $L^1(X)$. \vspace{0.5\baselineskip}
\end{corollary}

We prove the corollary in Section~\ref{sec:Const}. \vspace{0.5\baselineskip}

Finally, we comment on the hypothesis $\rho\geq c_1>0$.
If it is omitted, we can still get the following result:

\begin{corollary}\label{cor:DegeneratedDensity}
Let $p\geq 1$ and assume (A2), (B1-3) and (C1-4) are in force and that $\rho\in[0, c_2]$, for some $c_2<\infty$.
Set $X_+:=\{ x\in X \,:\, \rho(x)>0 \}$ and define the functional $\widetilde{\cG}_\infty^{(p)}: L^1(X)\to[0,+\infty]$ as
\[
\widetilde{\cG}_\infty^{(p)}(u):=
\left\{
\begin{array}{ll}
\displaystyle\int_{\partial^*\{u=1\}} \sigma^{(p)}(x,\nu_u(x)) \rho(x) \, \dd \cH^{d-1}(x) & \text{if }
    u\in BV_{loc}(X_+;\{\pm 1\})\,, \\
&\\
+\infty & \text{else}\,,
\end{array}
\right.
\]
where $BV_{loc}(X_+;\{\pm 1\})$ denotes the space of functions $u\in L^1(X;\{\pm1\})$ such that $u\in BV(K;\{\pm1\})$ for any compact set $K\subset X_+$.
Then, the following holds:
\begin{itemize}
\item (compactness) for any compact set $K\subset X_0$ we have that any sequence $\{u_n\}_{n=1}^\infty\subset L^1(K;\mu_n)$ satisfying $\sup_{n\in \bbN} \cG_n^{(p)}(u_n) < \infty$ is relatively compact in $TL^1$ and each cluster point $u$ has $\widetilde{\cG}_\infty^{(p)}(u)<\infty$;
\item ($\Gamma$-convergence) $\Glim_{n\to \infty}(TL^1) \cG_n^{(p)} = \widetilde{\cG}_\infty^{(p)}$.
\end{itemize}
\vspace{0\baselineskip}
\end{corollary}

We omit a rigorous proof of the corollary and just explain the details.
Indeed, for fixed a compact set $K\subset X_+$, the continuity of $\rho$ implies that $\min_{K}\rho\geq c_1>0$.
Thus the result of Corollary \ref{cor:DegeneratedDensity} follows by applying Theorem \ref{thm:MainRes:Compact&Gamma}.


\subsection{Related Works \label{subsec:Intro:LitRev}}

The functional $\cG^{(p)}_n$ in the case $p=1$ has been considered by the second author and Theil in \cite{thorpe17AAA}, where a similar $\Gamma$-convergence result has been proved. The difference is that, in the case $p=1$, the limit energy density function $\sigma^{(1)}$ can be given explicitly, via an integral.
In \cite{vangennip12a} van Gennip and Bertozzi studied the Ginzburg-Landau functional on 4-regular graphs for $d=2$ and $p=2$ proving limits for $\eps\to 0$ and $n\to \infty$ (both simultaneously and independently).

The $TL^p$ topology, as introduced by Garc\'{i}a Trillos and Slep\v{c}ev~\cite{garciatrillos16}, provides a notion of convergence upon which the $\Gamma$-convergence framework can be applied.
This method has now been applied in many works, see, for instance, \cite{garciatrillos16,thorpe17AAA,garciatrillos15aAAA,dunlop18,davis16AAA,garciatrillos15c,garciatrillos16bAAA,garciatrillos16aAAA,slepcev17}. 
Further studies on this topology can be found in 
\cite{garciatrillos16,garciatrillos15aAAA,thorpe17bAAA,thorpe17cAAA}.

The literature on phase transitions problems is quite extensive.
Here we just recall some of the main results, starting from the pioneering work \cite{modica77} of Modica and Mortola and of Mortola~\cite{modica87} (see also Sternberg \cite{sternberg88}), where the scalar isotropic case has been studied.
The vectorial case has been considered by Kohn and Sternberg in \cite{kohn89}, Fonseca and Tartar in \cite{fonseca89} and Baldo \cite{baldo}.
A study of the anisotropic case has been carried out by Bouchitt\'{e}~\cite{bouchitte90} and Owen~\cite{owen91} in the scalar case, and by Barroso and Fonseca~\cite{barroso94} and Fonseca and Popovici~\cite{fonpop} in the vectorial case.

Nonlocal approximations of local functionals of the perimeter type go back to the work \cite{alberti98} of Alberti and Bellettini (see also \cite{alberti98a}). Several variants and extensions have been considered since then (see, for instance, Savin and Valdinoci \cite{savvaldi} and Esedo\={g}lu and Otto \cite{eseotto}).
In particular, nonlocal functionals have been used by Brezis, Bourgain and Mironescu in \cite{brebourmir} to characterize Sobolev spaces (see also the work \cite{ponce04} of Ponce).

Approximations of (anisotropic) perimeter functionals via energies defined in the discrete setting have been carried out by Braides and Yip in \cite{brayip} and by Chambolle, Giacomini and Lussardi in \cite{chagialus}.

\subsection{The Choice of \texorpdfstring{$p$}{p}} \label{subsec:Intro:pEx}

By allowing for any $p\geq 1$ in the definition of our functionals $\cG_n^{(p)}$ (rather than $p=2$ usually considered) we allow for greater flexibility.
In particular the choice of $p$ has an important consequence regarding the balance between minimising the length of the perimeter and avoiding regions of high density.
As for Euclidean norms, as $p\to \infty$ the Dirichlet energy $\cE_n^{(p)}$, defined by
\[ \cE_n^{(p)}(u) = \ls \sum_{i,j=1}^n W_{ij}^p |u(x_i)-u(x_j)|^p\rs^{\frac{1}{p}}, \]
converges as $p\to \infty$ to $\cE_n^{(\infty)}$, defined by
\[ \cE_n^{(\infty)}(u) = \max_{i,j=1,\dots,n} W_{ij} |u(x_i)-u(x_j)|. \]
In fact the convergence also holds in a variational sense~\cite{egger90,kyng15}.
It is therefore of no surprise that the same holds in the Ginzburg-Landau setting (where the normalisation has been adjusted so that terms are of order 1 with respect to $p$).

\begin{proposition}
Fix $n\in \bbN$, and assume $V:\bbR\to [0,\infty)$ is continuous and weights $W_{ij}\geq 0$ for all $i,j=1,\dots, n$.
Let
\begin{align*}
\tilde{\cG}_n^{(p)}(u) & = \ls\frac{1}{\eps_n n^2} \sum_{i,j=1}^n W_{ij}^p |u(x_i)-u(x_j)|^p + \frac{1}{\eps_n^p n} \sum_{i=1}^n V^p(u(x_i)) \rs^{\frac{1}{p}} \\
\cG_n^{(\infty)}(u) & = \max\lb \max_{i,j=1,\dots,n} W_{ij} |u(x_i)-u(x_j)|, \frac{1}{\eps_n} \max_{i=1,\dots,n} V(u(x_i)) \rb.
\end{align*}
Then,
\[ \Glim_{p\to \infty} \tilde{\cG}_n^{(p)} = \cG_n^{(\infty)}. \]
\end{proposition}

\begin{proof}
The proof is easy since the functionals are on discrete domains.
In particular, for the liminf inequality we assume $L^0(X_n)\ni u_p \to u \in L^0(X_n)$ then,
\begin{align*}
\tilde{G}_n^{(p)}(u_p) & \geq \ls \frac{1}{\eps_n n^2} W_{rt}^p |u_p(x_r) - u_p(x_t)|^p + \frac{1}{\eps_n^p n} V^p(u_p(x_\ell)) \rs^{\frac{1}{p}} \\
 & \geq \max\lb \frac{1}{\eps_n^{\frac{1}{p}} n^{\frac{2}{p}}} W_{rt} |u_p(x_r) - u_p(x_t)|, \frac{1}{\eps_n n^{\frac{1}{p}}} V(u_p(x_\ell)) \rb \\
 & \to \max \lb W_{rt} |u(x_r) - u(x_t)|, \frac{1}{\eps_n} V(u(x_\ell)) \rb
\end{align*}
for any $r,t,\ell\in \{1,\dots, n\}$.
Choosing $(r,t) = \argmax_{i,j} W_{ij}|u(x_i)-u(x_j)|$ and $\ell = \argmax_i V(u(x_i))$ implies
\[ \liminf_{p\to \infty}\tilde{G}_n^{(p)}(u_p) \geq G_n^{(\infty)}(u). \]
For the recovery sequence we consider the sequence $u_p=u$ then,
\[ \tilde{G}_n^{(p)}(u_p) \leq \max\lb \frac{1}{\eps_n^{\frac{1}{p}}} \max_{i,j=1,\dots, n} W_{ij} |u(x_i) - u(x_j)|, \frac{1}{\eps_n} \max_{i=1,\dots,n} V(u(x_i)) \rb \to \cG_n^{(\infty)}(u) \]
as required.
\end{proof}

There are several ways we could have renormalised $\cG_n^{(p)}$ to obtain a well defined limit as $p\to \infty$.
We choose a normalisation $\tilde{\cG}_n^{(p)}$ that closely resembled $\cG_n^{(p)}$ and such that the limit $\cG_n^{(\infty)}$ will exhibit a phase transition as $n\to \infty$.
In particular, the limit functional $\cG_n^{(\infty)}$ is chosen so that minimisers $u_n$ have the property that $u_n(x)\to \{\pm 1\}$ for almost every $x$, and $\max_{i,j} W_{ij}|u_n(x_i)-u_n(x_j)| = O(1)$.

The important property to note is that the $p=\infty$ limit depends only on the data through its support.
In particular, if one considers the semi-supervised learning problem (similar to Remark~\ref{rem:Intro:MainRes:SLL} with $\lambda_n=\infty$):
\[ \text{minimise } \tilde{\cG}_n^{(p)}(u) \text{ subject to } u(x_i) = y_i \text{ for } x_i \in X^\prime \]
then the density of the data in $X\setminus X^\prime$ is irrelevant for $p=\infty$.

To illustrate this we consider a simple example.
We consider a density $\rho$ as in Figure~\ref{fig:Intro:p:ex}; in particular the density is given by $\rho(x) \propto \phi(x_1)\lfloor_X$ where $\phi$ is the density of a normal random variable, with mean at $0$ and standard deviation $0.25$, and $X$ is the bean shape.
The regions of labelled data are chosen to be the balls $B((-0.3,0),0.08)$ and $B((0.3,0),0.08)$, with the first ball labelled $-1$ and the second ball labelled $+1$.
The distribution is chosen so that the density is largest where the bean is narrowest.
The optimal partitioning, given by minimizing $\cG_n^{(p)}$ is therefore a trade off between minimizing the length of the perimeter of the partitioning and avoiding high density regions.
We plot the optimal partitioning for two choices of $p$, $p=2$ and $p=100$, using the same data set.
We see that for $p=2$ the partitioning avoids the region of high density at the cost of increasing the length of the boundary.
On the other hand, when $p=100$ the partitioning is much closer to the centre of the bean where the length of the partitioning is minimized.

\begin{figure}
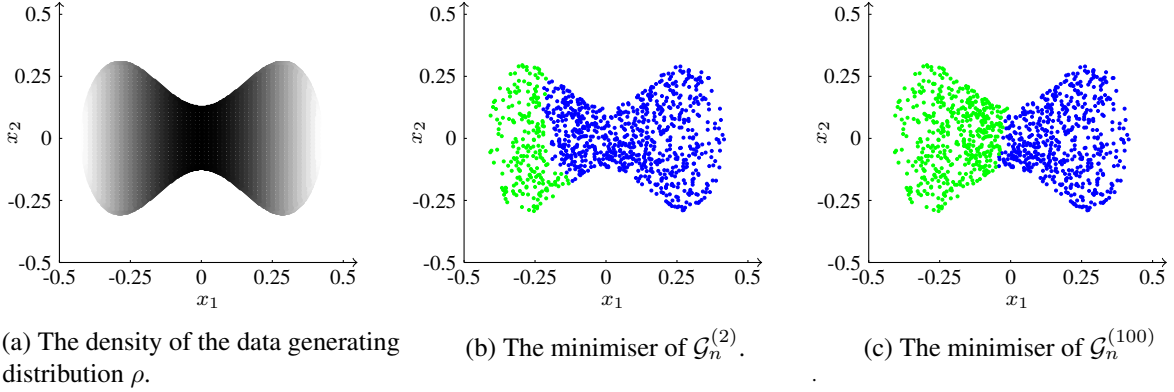

\centering
\setlength\figureheight{0.25\textwidth}
\setlength\figurewidth{0.32\textwidth}
\begin{subfigure}[t]{0.32\textwidth}
\scriptsize
\input{BeanDensityMS.tikz}
\caption{The density of the data generating distribution $\rho$.}
\end{subfigure}
\begin{subfigure}[t]{0.32\textwidth}
\scriptsize
\input{p2Cut_n1000_N10_s1_sd0p25.tikz}
\caption{The minimiser of $\cG_n^{(2)}$.}
\end{subfigure}
\begin{subfigure}[t]{0.32\textwidth}
\scriptsize
\input{p100Cut_n1000_N10_s1_sd0p25.tikz}
\caption{The minimiser of $\cG_n^{(100)}$}.
\end{subfigure}
\caption{How the choice of $p$ affects the trade-off between avoiding high density regions and minimising the length of the classification boundary.
The experiment setup is described in Section~\ref{subsec:Intro:pEx} and we computed minimisers using a gradient flow (known as the Allen-Cahn equation).
\label{fig:Intro:p:ex}
}
\end{figure}

\subsection{Application to Image Segmentation}  \label{subsec:Intro:AnisoEx}

Let us consider segmentation in images.
Let $x_i\in [0,1]^2$ denote the location of pixels (in particular $X_n = \{x_i\}_{i=1}^n$ form a regular grid with distance $\frac{1}{\sqrt{n}}$ between neighbouring pixels).
Let $y_i\in \bbR^3$ be the RGB values of a given image at the pixel located at $x_i$.
We assume $\cI\subset \{1,\dots,n\}$ indexes the a-priori labelled pixels and $f:\{x_i\}_{i\in\cI} \to \bbR$ are the given labels.
We wish to label the remaining pixels $\{x_i\}_{i\not\in \cI}$.
Define
\[ W_{i,j} = \frac{1}{\eps^2} \Phi(|y_{i} - y_{j}|) \chi_{{|x_{i} - x_{j}|\leq\eps}} \]
where $\Phi:\bbR\to\bbR$ is a decreasing function and $\eps$ determines the number of neighbouring pixels we connect, e.g. $\eps = \frac{1}{\sqrt{n}}$ connects the four neighbouring pixels, $\eps = \frac{\sqrt{2}}{\sqrt{n}}$ connects the eight (including diagonals) neighbouring pixels, etc.
Note that we are thinking of this problem in 2D where the kernel $\eta$ depends not just on the spatial difference but also on the values of pixels, i.e. $\eta(x_i-x_j) = \eta(x_i-x_j;y_i,y_j)$.
This falls slightly outside of our theoretical framework as the function $\eta$ now depends on more than just the difference between $x_i$ and $x_j$. 

The segmentation is defined by minimising
\[ \cG^{(p)}_n(u) + \cK_n(u) = \frac{\lambda}{2} \sum_{i\in \cI} |f(x_{i})- u(x_{j})|^2 + \frac{1}{\eps n^2} \sum_{i,j\in \{1,\dots,n\}} W_{ij} |u(x_{i}) - u(x_{j})|^p + \frac{1}{\eps n} \sum_{i=1}^n  V(u(x_{i})). \]
This setup is similar to Bertozzi and Flenner~\cite{bertozzi12} and
Calatroni, van Gennip, Sch\"{o}nlieb, Rowland and Flenner~\cite{calatroni17} however they embed the problem into a much higher dimensional space.
In particular, in~\cite{bertozzi12,calatroni17} the vertex corresponding to pixel $i$ is the vector consisting of the RGB values of the $M$ nearest neighbours, so that the graph is embedded in $\bbR^{3M}$ compared to the ambient dimension in our setup which is $\bbR^2$.
We refer to methods in~\cite{bertozzi12,calatroni17} as the segmentation in colour space method and our method as segmentation in pixel space.

It is not within the scope of this paper to comprehensively compare the differences in approach (in fact we note that the theoretical results in this paper apply to~\cite{bertozzi12,calatroni17}).
However, we point out that the method proposed here has the advantage that the parameter $\eps$ directly controls the number of pixels connected which can make it easier for parameter selection.
Choosing the parameter $\eps$ in the colour space segmentation method is not as clear as one needs to compute the connectivity radius of the graph.

On the other hand, the pixel space segmentation method proposed here is spatially local which can make it difficult for similar colours that are spatially separated to be placed in the same class.
In the colour space segmentation method similar colours are close on the graph regardless of the pixel locations.

In Figure~\ref{fig:Intro:AnisoEx:ex} we plot an example of an image segmented into background and foreground.
Given the image in Figure~\ref{fig:Intro:AnisoEx:ex}(a), where regions have been labelled either background (marked green), or foreground (marked red).
We construct a graph as described above with
\[ \Phi(t) = \exp\l -\frac{t^2}{\tau}\r, \qquad \tau = 5\times 10^{-4}  \]
and choose $\eps = \frac{1}{\sqrt{n}}$ and $p=2$.
Using a gradient flow to minimise $\cG_n^{(p)}+\cK_n$ we produce the output given in Figure~\ref{fig:Intro:AnisoEx:ex}(b).
Since $\eps$ determines the number of connections of each pixel it was easy to select parameters with a desirable output.

\begin{figure}
\centering
\begin{subfigure}[t]{0.48\textwidth}
\includegraphics[width=0.9\textwidth]{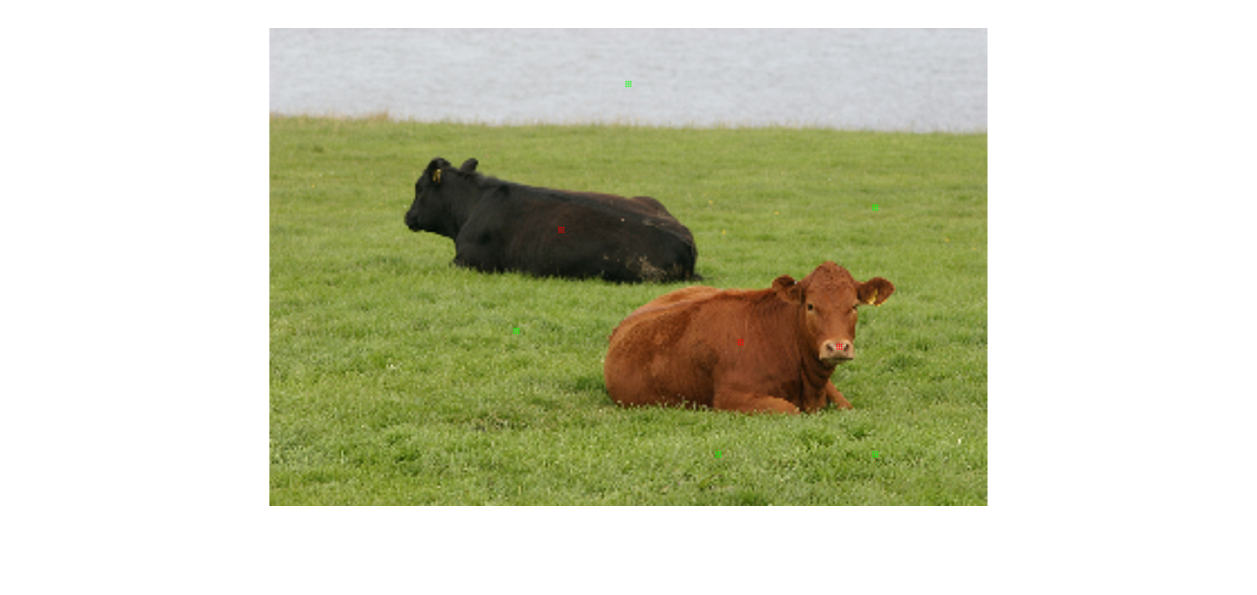}
\caption{Input images with a-priori given labels marked.}
\end{subfigure}
~
\begin{subfigure}[t]{0.48\textwidth}
\includegraphics[width=0.9\textwidth]{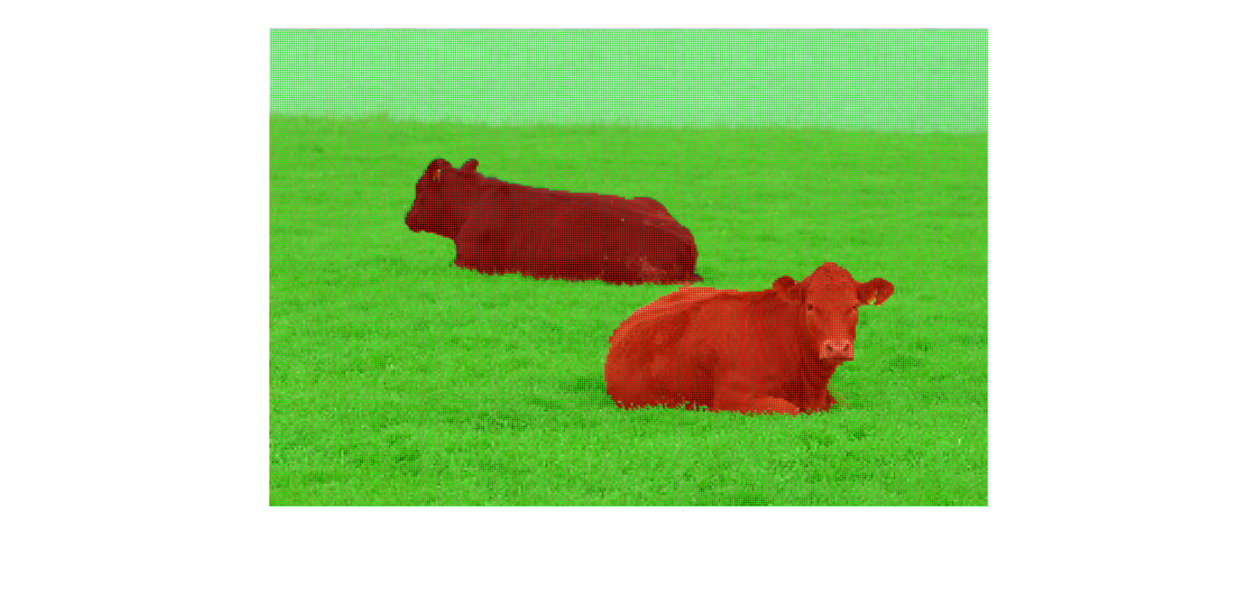}
\caption{The minimiser of $\cG_n^{(2)}+\cK_n$.}
\end{subfigure}
\caption{An image segmentation example with experimental setup as described in Section~\ref{subsec:Intro:AnisoEx}.
The image is from the Microsoft image database and also appeared in~\cite{bertozzi12}.
\label{fig:Intro:AnisoEx:ex}
}
\end{figure}


\section{Background \label{sec:Back}}

\subsection{Notation}

In the following $\chi_E$ will denote the characteristic function of a set $E\subset\bbR^d$, while
$\Vol(E)=\cL^d(E)$ will denote its $d$-dimensional Lebesgue measure and $\cH^{d-1}(E)$ its $(d-1)$-Hausdorff measure.
Moreover, with $B(x,r)$ we will denote the ball centered at $x\in\bbR^d$ with radius $r>0$
and we set $\mathbb{S}^{d-1}:=\partial B(0,1)$.
The identity map will be denoted by $\mathrm{Id}$.

Given an open set $X\subset\bbR^d$, we define the space
\[
\cP(X):=\{\, \text{ Radon measures } \mu \text{ on } X \text{with } \mu(X)=1 \,\}\,.
\]

Given a set of data points $\{x_i\}_{i=1}^n$ we define the empirical measure as follows.
\begin{mydef}\label{def:empmeas}
For all $n\in\bbN$ let $X_n:=\{x_i\}_{i=1}^n$ be a set of $n$ random variables. We define the \emph{empirical measure} $\mu_n$ as
\[
\mu_n:=\frac{1}{n}\sum_{i=1}^n \delta_{x_i}\,,
\]
where $\delta_x$ denotes the Dirac delta centered at $x$. 
\vspace{0.5\baselineskip}
\end{mydef}

We state our results in terms of a general sequence of empirical measures $\mu_n$ that converge weak$^*$ to some $\mu\in\cP(X)$.
An important special case is when $x_i$ are independent and identically distributed (which we abbreviate to iid) from $\mu$.

\begin{remark}\label{rem:weakconv}
When $x_i\iid \mu$ then $\mu_n$ converges, with probability one, to $\mu$ weakly${^*}$ in the sense of measures, see for example~\cite[Theorem 11.4.1]{dudley02}, (and we write $\mu_n\weakstarto\mu$), \emph{i.e.}
\[ \int_X \varphi \,\dd \mu_n\to \int_X \varphi\,\dd\mu \]
as $n\to\infty$, for all $\varphi\in C_c(X)$.
\end{remark}

We write $L^p(X,\mu;Y)$ for the space of $L^p$ integrable, with respect to $\mu$, functions from $X$ to $Y$.
We will often suppress the $Y$ dependence and just write $L^p(X,\mu)$.
Moreover, if $\mu=\cL^d$ then we will often write $L^p(X) = L^p(X,\mu)$.
If $\mu=\mu_n$ is the empirical measure we also write $L^p(X_n) = L^p(X,\mu_n)$.


\subsection{Transportation theory \label{subsec:Prelim:Trans}}

In this section we collect the material needed in order to explain how to compare functions defined in different spaces, namely a function $w\in L^1(X,\mu)$ and a function $u\in L^1(X_n,\mu_n)$, where $X\subset\bbR^d$ is an open set and $X_n\subset X$ is a finite set of points.
This is fundamental in stating our $\Gamma$-convergence result (Theorem \ref{thm:MainRes:Compact&Gamma}).
The $TL^p$ space was introduced in~\cite{garciatrillos16} and consists of comparing $w$ and a piecewise constant extension of the function $u$ in $L^p$.
In particular, we take a map $T:X\to X_n$ and we consider the function $v:X\to\bbR$ defined as $v:=u\circ T$. 
In order that this defines a metric one needs to impose conditions on $T$, the natural conditions are that $T$ "matches the measure $\mu$ with $\mu_n$" and is \emph{optimal} in the sense that matching moves as little mass as possible (see Theorem \ref{thm:optrate}).
This will be done by using the optimal transport distance that we recall now (see also~\cite{villani} for background on optimal transport and \cite{garciatrillos16,thorpe17cAAA} for a further description of the $TL^p$ space).

\begin{mydef}
\label{def:Back:Trans:Wass}
Let $X\subset\bbR^d$ be an open set and let $\mu,\lambda\in\cP(X)$.
We define the set of couplings $\Gamma(\mu,\lambda)$ between $\mu$ and $\lambda$ as
\[
\Gamma(\mu,\lambda):=\left\{\, \pi\in\cP(X\times X) \,:\, \pi(A\times X)=\mu(A),\, \pi(X\times A)=\lambda(A),\, \text{ for all measurable } A\subset X  \,\right\}\,.
\]
For $p\in[1,+\infty]$, we define the $p$-Wasserstein distance between $\mu$ and $\lambda$ as follows:
\begin{itemize}
\item when $1\leq p<\infty$,
\[
\dist_p(\mu,\lambda):=\inf\left\{\, \l\int_{X\times X} |x-y|^p \, \mathrm{d}\pi(x,y)\,\r^{\frac{1}{p}} \,:\, \pi\in\Gamma(\mu,\lambda)  \,\right\}\,,
\]
\item when $p=\infty$,
\[
\dist_\infty(\mu,\lambda):=\inf\left\{\, \mathrm{esssup}_\pi \left\{\, |x-y| \,:\, (x,y)\in X\times X \,\right\}\, : \, \pi\in\Gamma(\mu,\lambda)  \,\right\}\,,
\]
where $\mathrm{esssup}_\pi$ denotes the essential supremum with respect to the measure $\pi$.
\end{itemize} \vspace*{0\baselineskip}
\end{mydef}

\begin{remark}\label{rem:ot}
The infimum problems in the above definition are known as the Kantorovich optimal transport problems and the distance is commonly called the $p^\text{th}$ Wasserstein distance or sometimes the earth movers distance.
It is possible to see (see \cite{villani}) that the infimum is actually achieved.
Moreover, the metric $\dist_p$ is equivalent to the weak$^*$ convergence of probability measures $\cP(X)$ (plus convergence of $p^{\text{th}}$ moments).
\vspace{0.8\baselineskip}
\end{remark}

We now consider the case we are interested in: take $\mu\in\cP(X)$ with $\mu=\rho\cL^d$ (where $\cL^d$ is the $d$-dimensional Lebesgue measure on $\bbR^d$) and assume the density $\rho$ is such that $0<c_1\leq \rho\leq c_2<\infty$.
In this case the Kantorovich optimal transport problem is equivalent to the Monge optimal transport problem (see \cite{gangmcc}).
In particular, for $p\in[1,+\infty)$ it holds that
\[
\dist_p(\mu,\lambda)=\min\left\{\, \|\mathrm{Id}-T\|_{L^p(X,\mu)} \,:\, T:X\to X \text{ Borel},\,\, T_\# \mu=\lambda  \,\right\}\,,
\]
where
\[
\|\mathrm{Id}-T\|^p_{L^p(X,\mu)}:=\int_X |x-T(x)|^p \rho(x)\,\dd x
\]
and we define the \emph{push forward} measure $T_\#\mu\in\cP(X)$ as $T_\#\mu(A):=\mu\left( T^{-1}(A) \right)$
for all $A\subset X$.
In the case $p=+\infty$ we get
\[
\dist_\infty(\mu,\lambda)=\inf\left\{\, \|\mathrm{Id}-T\|_{L^\infty(X,\mu)} \,:\, T:X\to X \text{ Borel},\,\, T_\# \mu=\lambda  \,\right\}\,,
\]
where
\[ \|\mathrm{Id}-T\|_{L^\infty(X,\mu)}:=\esssup_{X}\rho(x)[x-T(x)]. \]
A map $T$ is called a \emph{transport map} between $\mu$ and $\lambda$ if $T_{\#}\mu=\lambda$.

Throughout the paper we will assume the empirical measures $\mu_n$ converges weakly${^*}$ to $\mu$ (see Remark \ref{rem:weakconv} for iid samples) so by Remark~\ref{rem:ot} there exists a sequence of Borel maps $\{T_n\}_{n=1}^\infty$ with $T_n: X\to X_n$ and $(T_n)_\#\mu=\mu_n$ such that
\[
\lim_{n\to\infty}\|\mathrm{Id}-T_n\|^p_{L^p(X,\mu)}=0\,.
\]
Such a sequence of functions $\{T_n\}_{n=1}^\infty$ will be called \emph{stagnating}.
We are now in position to define the notion of convergence for sequences $u_n\in L^p(X_n)$ to a continuum limit $u\in L^p(X,\mu)$.

\begin{mydef} \label{def:Back:Trans:TLpConv}
Let $u_n\in L^p(X_n)$, $w\in L^p(X,\mu)$ where $X_n=\{x_i\}_{i=1}^n$ and assume that the empirical measure $\mu_n$ converges weak$^*$ to $\mu$. We say that $u_n\to w$ in $TL^p(X)$, and we write $u_n\TLpto w$, if there exists a sequence of stagnating transport maps $\{T_n\}_{n=1}^\infty$ between $\mu$ and $\mu_n$ such that
\begin{equation}\label{eq:stagn}
\| v_n-w\|_{L^p(X,\mu)}\to0\,,
\end{equation}
as $n\to\infty$, where $v_n:=u_n\circ T_n$.
\vspace{0.5\baselineskip}
\end{mydef}

\begin{remark}
One can show that if \eqref{eq:stagn} holds for one sequence of stagnating maps, then it holds for all sequences of stagnating maps\cite[Proposition 3.12]{garciatrillos16}. Moreover, since $\rho$ is bounded above and below it holds
\[
\| v_n-\mathrm{Id}\|_{L^p(X,\mu)}\to0\quad\Leftrightarrow\quad \| v_n-\mathrm{Id}\|_{L^p(X)}\to0\,.
\]
\vspace{0.8\baselineskip}
\end{remark}

We have introduced $TL^p$ convergence $u_n\TLpto u$ by defining transport maps $T_n:X\to X_n$ which "optimally partition" the space $X$ after which we define a piecewise constant extension of $u_n$ to the whole of $X$.
This constructionist approach is how we use $TL^p$ convergence in our proofs.
However, this description hides the metric properties of $TL^p$.
We briefly mention here the metric structure which characterises the convergence given in Definition~\ref{def:Back:Trans:TLpConv}.
We define the $TL^p(X)$ space as the space of couplings $(u,\mu)$ where $\mu\in \cP(X)$ has finite $p^\text{th}$ moment and $u\in L^p(\mu)$.
We define the distance $d_{TL^p}:TL^p(X)\times TL^p(X)\to [0,+\infty)$ for $p\in [1,+\infty)$ by
\begin{align*}
d_{TL^p}((u,\mu),(v,\lambda)) & := \min_{\pi\in \Gamma(\mu,\nu)} \l\int_{X^2} |x-y|^p + |u(x) - v(y)|^p \, \dd \pi(x,y)\r^{\frac{1}{p}} \\
 & = \inf_{T_{\#}\mu=\lambda} \l\int_X |x-T(x)|^p + |u(x)-v(T(x))|^p \, \dd \mu(x)\r^{\frac{1}{p}},
\end{align*}
or for $p=+\infty$ by
\begin{align*}
d_{TL^\infty}((u,\mu),(v,\lambda)) & := \inf_{\pi\in \Gamma(\mu,\nu)} \l \essinf_\pi \lb |x-y| + |u(x) - v(y)| \, : \,  (x,y)\in X\times X \rb \r \\
 & = \inf_{T_{\#}\mu=\lambda} \l \essinf_\mu \lb |x-T(x)| + |u(x)-v(T(x))| \, : \, x\in X \rb \r.
\end{align*}

\begin{proposition}
The distance $d_{TL^p}$ is a metric and furthermore, $d_{TL^p}((u_n,\mu_n),(u,\mu))\to 0$ if and only if $\mu_n\weakstarto \mu$ and there exists a sequence of stagnating transport maps $\{T_n\}_{n=1}^\infty$ between $\mu$ and $\mu_n$ such that $\|u_n\circ T_n-u\|_{L^p(X,\mu)}\to 0$.
\end{proposition}

The proof is given in~\cite[Remark 3.4 and Proposition 3.12]{garciatrillos16}.
Note that Definition~\ref{def:Back:Trans:TLpConv} characterises $TL^p$ convergence.

In order to be able to write the discrete functional  we will need the following result.

\begin{lemma}\label{lem:writeint}
Let $\lambda\in\cP(X)$ and let $T:X\to X$ be a Borel map. Then, for any $u\in L^1(X,\lambda)$ it holds
\[
\int_X u\,\dd T_\#\lambda = \int_X u\circ T\,\dd\lambda\,.
\]
\end{lemma}

\begin{proof}
Let $s:X\to\bbR$ be a simple function. Write
\[
s=\sum_{i=1}^k a_i \chi_{U_i}\,.
\]
Then
\[
\int_X s\,\dd T_\#\lambda =\sum_{i=1}^k a_i T_\#\lambda(U_i)=\sum_{i=1}^k a_i \lambda\left(T^{-1}(U_i)\right)=\int_X s\circ T \,\dd \lambda\,.
\]
The result then follows directly from the definition of the integral.
\end{proof}

\begin{remark}\label{rem:writeint}
Applying the above result to the empirical measures $\mu_n=\frac{1}{n}\sum_{i=1}^n \delta_{x_i}$ and $u\in L^1(X_n)$ we get
\[
\frac{1}{n}\sum_{i=1}^n u(x_i)=\int_X v_n(x) \,\dd\mu(x)\,,
\]
where $v_n:=u\circ T_n$ for any $T_n$ such that $(T_n)_\#\mu=\mu_n$. \vspace{0.8\baselineskip}
\end{remark}

In \cite{garciatrillos15} the authors, Garc\'{i}a Trillos and Slep\v{c}ev, obtain the following rate of convergence for a sequence of stagnating maps. This is of crucial importance for applying the results of this paper to the iid setting.

\begin{theorem}\label{thm:optrate}
Let $X\subset\bbR^d$ be a bounded, connected and open set with Lipschitz boundary.
Let $\mu\in\cP(X)$ be of the form $\mu=\rho\cL^d$ with $0<c_1\leq\rho\leq c_2<\infty$.
Let $\{x_i\}_{i=1}^\infty$ be a sequence of independent and identically distributed random variables distributed on $X$ according to the measure $\mu$, and let $\mu_n$ be the associated empirical measure.
Then, there exists a constant $C>0$ such that, with probability one, there exists a sequence $\{T_n\}_{n=1}^\infty$ of maps $T_n:X\to X$ with $(T_n)_\#\mu=\mu_n$ and
\[
\limsup_{n\to\infty}\frac{\| T_n-\mathrm{Id}\|_{L^\infty(X)}}{\delta_n}\leq C\,,
\]
where
\[
\delta_n:=
\lb \begin{array}{ll}
\sqrt{\frac{\log\log n}{n}} & \text{if } d=1\,, \\
\frac{(\log n)^{\frac{3}{4}}}{\sqrt{n}} & \text{if } d=2\,, \\
\l \frac{\log n}{n} \r^{\frac{1}{d}} & \text{if } d\geq 3\,.
\end{array}
\rd
\]
\vspace{0.5\baselineskip}
\end{theorem}

\begin{remark}
The proof for $d=1$ is simpler and follows from the law of iterated logarithms.
Notice that the connectedness of $X$ is essential in order to get the above result. \vspace{0.8\baselineskip}
\end{remark}

By the above theorem our main result, Theorem~\ref{thm:MainRes:Compact&Gamma}, holds with probability one when $x_i\iid \mu$ and the graph weights are scaled by $\eps_n$ with $\eps_n\gg \delta_n$.


\subsection{Sets of finite perimeter}

In this section we recall the definition and basic facts about sets of finite perimeter. We refer the reader to \cite{AFP} for more details. 

\begin{mydef}
Let $E\subset\bbR^d$ with $\Vol(E)<\infty$ and let $X\subset\bbR^d$ be an open set.
We say that $E$ has \emph{finite perimeter} in $X$ if
\[
|D\chi_E|(X):=\sup\left\{\, \int_E \mathrm{div}\varphi \,\dd x \,:\, \varphi\in C^1_c(X;\bbR^d)\,,\, \|\varphi\|_{L^\infty}\leq1  \,\right\}<\infty\,.
\]
\vspace{0\baselineskip}
\end{mydef}

\begin{remark}\label{rem:defvar}
If $E\subset\bbR^d$ is a set of finite perimeter in $X$ it is possible to define a finite vector valued Radon measure $D\chi_E$ on $A$ such that
\[
\int_{\bbR^d} \varphi \,\dd D\chi_E=\int_E \mathrm{div}\varphi \,\dd x
\]
for all $\varphi\in C^1_c(X;\bbR^d)$. \vspace{0.8\baselineskip}
\end{remark}

\begin{mydef}
Let $X\subset\bbR^d$ be an open set and let $u\in L^1(X;\{\pm1\})$ with $\|u\|_{L^1(X)}<\infty$.
We say that $u$ is of \emph{bounded variation} in $X$, and we write $u\in BV(X;\{\pm1\}),$ if $\{u=1\}:=\{ x\in X \,:\, u(x)=1\}$ has finite perimeter in $X$. \vspace{0.5\baselineskip}
\end{mydef}

\begin{mydef}\label{eq:defnormal}
Let $E\subset\bbR^d$ be a set of finite perimeter in the open set $X\subset\bbR^d$. We define $\partial^* E$, the \emph{reduced boundary} of $E$, as the set of points $x\in\bbR^d$ for which the limit
\[
\nu_E(x):=-\lim_{r\to0}\frac{D\chi_E(x+rQ)}{|D\chi_E|(x+rQ)}
\]
exists and is such that $|\nu_E(x)|=1$. Here $Q$ denotes the unit cube of $\bbR^d$ centered at the origin with sides parallel to the coordinate axes. The vector $\nu_E(x)$ is called the \emph{measure theoretic exterior normal} to $E$ at~$x$. \vspace{0.5\baselineskip}
\end{mydef}

We now recall the structure theorem for sets of finite perimeter due to De Giorgi, see~\cite[Theorem 3.59]{AFP} for a proof of the following theorem.

\begin{theorem}\label{thm:DeGiorgi}
Let $E\subset\bbR^d$ be a set with finite perimeter in the open set $X\subset\bbR^d$.
Then
\begin{itemize}
\item[(i)] for all $x\in\partial^* E$ the set $E_r:=\frac{E-x}{r}$ converges locally in $L^1(\bbR^d)$ as $r\to0$ to the 
    halfspace orthogonal to $\nu_E(x)$ and not containing $\nu_E(x)$,
\item[(ii)] $D\chi_E=\nu_E\,\cH^{d-1}\res\partial^* E$,
\item[(iii)] the reduced boundary $\partial^* E$ is $\cH^{d-1}$-rectifiable, \emph{i.e.},
there exists Lipschitz functions $f_i:\bbR^{d-1}\to\bbR^d$ such that
\[
\partial^* E=\bigcup_{i=1}^\infty f_i(K_i)\,,
\]
where each $K_i\subset\bbR^{d-1}$ is a compact set.
\end{itemize}
\vspace{0\baselineskip}
\end{theorem}

\begin{remark}\label{rem:newdefnormal}
Using the above result it is possible to prove that (see \cite{fonsecamuller93})
\[
\nu_E(x)=-\lim_{r\to0}\frac{D\chi_E(x+rQ)}{r^{d-1}}
\]
for all $x\in\partial^* E$, where $Q$ is a unit cube centred at $0$ with sides parallel to the co-ordinate axis.
\end{remark}

The construction of the recovery sequences in Section \ref{subsec:ConvNLContinuum:Limsup} and Section \ref{subsec:ConvGraph:Limsup} will be done for a special class of functions, that we introduce now.

\begin{mydef}\label{def:polyfun}
We say that a function $u\in L^1(X;\{\pm1\})$ is \emph{polyhedral} if $u=\chi_E-\chi_{X\setminus E}$, where $E\subset X$ is a set whose boundary is a Lipschitz manifold contained in the union of finitely many affine hyperplanes. In particular, $u\in BV(X,\{\pm1\})$. \vspace{0.5\baselineskip}
\end{mydef}

Using the result \cite[Theorem 3.42]{AFP} and the fact that it is possible to approximate every smooth surface with polyhedral sets, it is possible to obtain the following density result.

\begin{theorem}\label{thm:denspoly}
Let $u\in BV(X;\{\pm1\})$. Then there exists a sequence $\{u_n\}_{n=1}^\infty\subset BV(X;\{\pm1\})$ of polyhedral functions such that $u_n\to u$ in $L^1(X)$ and $|Du_n|(X)\to|Du|(X)$. In particular $D u_n\stackrel{{w}^*}{\rightharpoonup}Du$. \vspace{0.5\baselineskip}
\end{theorem}

Finally, we recall a result due to Reshetnvyak in the form we will need in this paper (for a proof of the general case see, for instance, \cite[Theorem 2.38]{AFP}).

\begin{theorem}\label{thm:rese}
Let $\{E_n\}_{n=1}^\infty$ be a sequence of sets of finite perimeter in the open set $X\subset\bbR^d$ such that $D\chi_{E_n}\stackrel{{w}^*}{\rightharpoonup} D\chi_E$ and $|D\chi_{E_n}|(X)\to|D\chi_{E}|(X)$, where $E$ is a set of finite perimeter in $X$.
Let $f:X\times \mathbb{S}^{d-1}\to[0,\infty)$ be an upper semi-continuous function. Then
\[
\limsup_{n\to\infty}\int_{\partial^* E_n\cap X} f\left(\,x,\nu_{E_n}(x)\right)\,\dd\cH^{d-1}(x)
    \leq \int_{\partial^* E\cap X} f\left(\,x,\nu_E(x)\right)\,\dd\cH^{d-1}(x)\,.
\]
\end{theorem}

\begin{remark}
A set of finite perimeter can have a very wild reduced boundary. Since the limiting energy density $\sigma^{(p)}$ depends on the exterior theoretic normal $\nu_u$, it would be very difficult to build an explicit recovery sequence for a set $\{u=1\}$ that is only assumed to have finite perimeter.
For this reason, we consider the family of polyhedral sets, for which we provide an explicit construction, and then we use the density result Theorem \ref{thm:denspoly} together with Theorem \ref{thm:rese} to address the general case.
\end{remark}


\subsection{\texorpdfstring{$\Gamma$}{Gamma}-convergence}\label{sec:Gammaconv}

We recall the basic notions and properties of $\Gamma$-convergence (in metric spaces) we will use in the paper (for a reference, see \cite{braides02,dalmaso93}).

\begin{mydef}
Let $(A,\mathrm{d})$ be a metric space. We say that $F_n:A\to[-\infty,+\infty]$ $\Gamma$-converges to $F:A\to[-\infty,+\infty]$,
and we write $F_n\Gammato F$ or $F=\Glim (d)_{n\to \infty} F_n$, if the following hold true:
\begin{itemize}
\item[(i)] for every $x\in A$ and every $x_n\to x$ we have
\[
F(x)\leq\liminf_{n\to\infty} F_n(x_n)\,;
\]
\item[(ii)] for every $x\in A$ there exists $\{x_n\}_{n=1}^\infty\subset A$ (the so called \emph{recovery sequence}) with $x\to x$ such that
\[
\limsup_{n\to\infty} F_n(x_n)\leq F(x)\,.
\]
\end{itemize} \vspace{0\baselineskip}
\end{mydef}

With a small abuse of notation we will write $\Glim(L^1)$ and $\Glim(TL^1)$ for $\Gamma$-convergence with respect to the $L^1$ metric $d_{L^1}:L^1(X)\times L^1(X)\to [0,\infty)$ and the $TL^1$ metric $d_{TL^1}:TL^1(X)\times TL^1(X)\to [0,\infty)$.

The notion of $\Gamma$-convergence has been designed in order for the following convergence of minimisers and minima result to hold (see for example~\cite{braides02,dalmaso93}).

\begin{theorem}\label{thm:convmin}
Let $(A,\mathrm{d})$ be a metric space and let $F_n\stackrel{\Gamma-(\mathrm{d})}{\longrightarrow} F$, where $F_n$ and $F$ are as in the above definition.
Let $\{\eps_n\}_{n=1}^\infty$ with $\eps_n\to0^+$ as $n\to\infty$ and let $x_n\in A$ be \emph{$\eps_n$-minimizers} for $F_n$, that is
\begin{equation}\label{eq:infFn}
F_n(x_n)\leq \max\left\{\, \inf_A F_n + \frac{1}{\eps_n}\,,\, -\frac{1}{\eps_n} \,\right\}\,.
\end{equation}
Then every cluster point of $\{x_n\}_{n=1}^\infty$ is a minimizer of $F$. \vspace{0.5\baselineskip}
\end{theorem}

\begin{remark}
The condition defining an $\eps$-minimizer takes into account the fact that the infimum of the functional can be $-\infty$. In the case $\inf_{n\in\bbN}\inf_{A}F_n>-\infty$, condition \eqref{eq:infFn} reduces, for $n$ sufficiently large, to
$F_n(x_n)\leq \inf_A F_n + \frac{1}{\eps_n}$.
\vspace{0.8\baselineskip}
\end{remark}

In the context of this paper we apply Theorem~\ref{thm:convmin} in order to prove Corollary~\ref{cor:MainRes:Constrained}.
In particular, we show that $\cG_n^{(p)}+\cK_n$ $\Gamma$-converges to $\cG_\infty+\cK_\infty$ and satisfies a compactness property.
We note that in general
\[
\Glim_{n\to \infty}(d) (F_n+G_n) \neq \Glim_{n\to \infty}(d) F_n + \Glim_{n\to \infty}(d) G_n.
\]
However, with a suitable strong notion of convergence of $G_n\to G$ we can infer the additivity of $\Gamma$-limits.

\begin{proposition}
\label{prop:contconv2}
Let $(A,d)$ be a metric space and let $F_n\Gammato F$.
Assume $G_n(u_n)\to G(u)$ and $G(u)>-\infty$ for any sequence $u_n\to u$ with $\sup_{n\in \bbN} F_n(u_n)<+\infty$ and $F(u)<+\infty$ then
\[ \Glim_{n\to \infty}(d) (F_n+G_n) = \Glim_{n\to \infty}(d) F_n + \Glim_{n\to \infty}(d) G_n. \]
\end{proposition}

The assumption in the above proposition is similar to the notion of continuous convergence, see~\cite[Definition 4.7 and Proposition 6.20]{dalmaso93}.
In our context continuous convergence is not quite the right concept, indeed the fidelity terms $\cK_n$ defined in Definition \ref{def:fidelityterm} do not continuously converge to $\cK_\infty$.
However we show, in Section~\ref{sec:Const}, that $\cK_n$ does satisfy the assumptions in Proposition~\ref{prop:contconv2}.


\section{Convergence of the Non-Local Continuum Model \label{sec:ConvNLContinuum}}

We first introduce the intermediary functional $\cF_\eps^{(p)}$ that is a non-local continuum approximation of the discrete functional $\cG_n^{(p)}$.

\begin{mydef}
Let $p\geq1$, $\eps>0$, $s_\eps>0$, and let $A\subset X$ be an open and bounded set. Define the functional $\cF^{(p)}_\eps(\cdot,A):L^1(X)\to[0,\infty]$ by
\begin{equation} \label{eq:ConvNLContinuum:Feps}
\cF^{(p)}_\eps(u,A) = \frac{s_\eps}{\eps} \int_{A\times A} \eta_\eps(x-z) |u(x)-u(z)|^p \rho(x)\rho(z) \, \dd x \, \dd z + \frac{1}{\eps} \int_A V(u(x)) \rho(x) \, \dd x\,.
\end{equation}
When $A=X$, we will simply write $\cF^{(p)}_\eps(u)$. \vspace{0.5\baselineskip}
\end{mydef}

This section is devoted to proving the following result.

\begin{theorem}
\label{thm:ConvNLContinuum:ConvNLContinuum}
Let $p\geq 1$, and assume $s_{\eps}\to1$ as $\eps\to0$.
Under conditions (A1), (B1-3) and (C1-3) the following holds:
\begin{itemize}
\item (Compactness) Let $\eps_n\to 0^+$ and $u_n\in L^1(X,\mu)$ satisfy $\sup_{n\in \bbN} \cF_{\eps_n}^{(p)}(u_n)<\infty$, then $u_n$ is relatively compact in $L^1(X,\mu)$ and each cluster point $u$ has $\cG_\infty^{(p)}(u)<\infty$;
\item ($\Gamma$-convergence) $\Glim_{\eps\to 0}(L^1) \cF_\eps^{(p)} = \cG_\infty^{(p)}$ and furthermore, if $u\in L^1(X,\mu)$ is a polyhedral function then for any $\zeta_\eps\to 0$ the recovery sequence $u_\eps$ can be chosen to satisfy the following: each $u_\eps$ is Lipschitz continuous with $\Lip(u_\eps) = \frac{1}{\zeta_\eps\eps}$ and $u_\eps(x) = u(x)$ for all $x$ satisfying $\dist(x,\partial^*\{u=1\}) > \frac{\eps}{\zeta_\eps}$. 
\end{itemize} \vspace{0\baselineskip}
\end{theorem}

See Section~\ref{subsec:ConvNLContinuum:Compact} for compactness and Sections~\ref{subsec:ConvNLContinuum:Liminf} and~\ref{subsec:ConvNLContinuum:Limsup} for the $\Gamma$-convergence.

The result of Theorem \ref{thm:ConvNLContinuum:ConvNLContinuum} is a generalization of a result by Alberti and Bellettini (see \cite{alberti98}) that we are going to recall for the reader's convenience.
First, we introduce notation for the special case when $p=2$ and $\rho\equiv1$ on $X$.

\begin{mydef}
Let $\eps>0$, and define the functionals $\cE_\eps, \cE_0: L^1(X)\to[0,\infty]$ by
\begin{align*}
\cE_\eps(u) & := \frac{1}{\eps} \int_X\int_X \eta_\eps(x-z) |u(x) - u(z)|^2 \, \dd x \, \dd z + \frac{1}{\eps} \int_X V(u(x)) \, \dd x 
\\
& \nonumber\\
\cE_0(u) & := \lb \begin{array}{ll} \displaystyle\int_{\partial^*\{u=1\}} \hat{\sigma}(\nu_u(x)) \, \dd \cH^{d-1}(x) & \text{if } u\in BV(X,\{\pm 1\})\,, \\
&\\
+ \infty & \text{else}\,,
\end{array} \rd \notag
\end{align*}
respectively, where
\begin{align*}
\hat{\sigma}(\nu) & := \min\lb E(f;\nu) \, | \, f:\bbR\to \bbR, \lim_{t\to \infty} f(t) = 1, \lim_{t\to-\infty} f(t) = -1 \rb \\
E(f;\nu) & := \int_{-\infty}^\infty \int_{\bbR^d} \eta(h) |f(t+h_\nu) - f(t)|^2 \, \dd h \, \dd t + \int_{-\infty}^\infty V(f(t)) \, \dd t \nonumber
\end{align*}
and $h_\nu := h \cdot \nu$. \vspace{0.5\baselineskip}
\end{mydef}

\begin{remark}
We make the following observations.
\begin{enumerate}
\item The fact that there exists a minimizer of $\hat{\sigma}(\nu)$ follows from~\cite[Theorem 2.4]{alberti98a}.
\item The minimum in $\hat{\sigma}(\nu)$ can be taken over non-decreasing functions with $tf(t)\geq 0$ for all $t\in \bbR$.
\item If $\eta$ is isotropic, \emph{i.e.}, $\eta(h)=\eta(|h|)$, then $E(f;\nu)$ is independent of $\nu$ and therefore $\sigma(\nu)$ is a constant and the functional $\cE_0$ is a multiple of the perimeter $\cH^{d-1}(\partial^*\{u=1\})$.
\end{enumerate} \vspace{0\baselineskip}
\end{remark}

The following theorem summerizes two results, namely ~\cite[Theorem 1.4]{alberti98} and~\cite[Theorem 3.3]{alberti98a}.

\begin{theorem}
\label{thm:ConvNLContinuum:AB:AB}
Assume that $X\subset\bbR^d$ is open, $V$ satisfies conditions (B1-3), and $\eta$ satisfies conditions (C1-3).
Then the following holds:
\begin{itemize}
\item (Compactness) Any sequence $\eps_n\to 0^+$ and $\{u_n\}\subset L^1(X)$ with $\sup_{n\in \bbN}\cE_{\eps_n}(u_n)<\infty$ is relatively compact in $L^1(X)$, furthermore any cluster point $u$ satisfies $\cE_0(u)<\infty$;
\item ($\Gamma$-convergence) $\Glim_{\eps\to 0}(L^1) \cE_\eps = \cE_0$.
\end{itemize} \vspace{0\baselineskip}
\end{theorem}

In fact the above theorem is true if (C1) is relaxed to $\eta\geq 0$ (i.e. positivity and continuity at the origin are not needed).  
The proof of Theorem \ref{thm:ConvNLContinuum:ConvNLContinuum} is based on the proof of \cite[Theorem 1.4]{alberti98}, where we 
have to deal with the fact that, in our case, we are considering a generic exponent $p\geq 1$, and that we have a density $\rho$.
The compactness proof is in Section~\ref{subsec:ConvNLContinuum:Compact} and the $\Gamma$-convergence in Sections~\ref{subsec:ConvNLContinuum:Liminf}-\ref{subsec:ConvNLContinuum:Limsup}.


\subsection{Compactness \label{subsec:ConvNLContinuum:Compact}}

The aim of this section is to show that any sequence $\{u_n\}_{n=1}^\infty\subset L^1(X,\mu)$ with $\sup_{n\in \bbN}\cF_{\eps_n}^{(p)}(u_n) < \infty$ is relatively compact in $L^1(X,\mu)$ and that $\cG_\infty^{(p)}(u)<\infty$ for any cluster point $u\in L^1(X,\mu)$.
This will prove the first part of Theorem~\ref{thm:ConvNLContinuum:ConvNLContinuum}.

The strategy of the proof is to apply the Alberti and Bellettini compactness result in Theorem~\ref{thm:ConvNLContinuum:AB:AB}.
When $p=2$ this follows from the upper and lower bounds on $\rho$ that imply an `equivalence' between $\cE_{\eps_n}$ and $\cF_{\eps_n}^{(2)}$.
When $p\neq 2$ we approximate $u_n$ with a sequence $v_n$ satisfying $v_n(x)\in\{\pm 1\}$ then since
$|v_n(x)-v_n(z)|^2=2^{2-p} |v_n(x)-v_n(z)|^p$ we can easily find an equivalence between $\cE_{\eps_n}$ and $\cF_{\eps_n}^{(p)}$.
We start with the preliminary result that shows $\cE_0(u)<\infty \Leftrightarrow \cG_\infty^{(p)}(u)<\infty$.

\begin{proposition}
\label{prop:ConvNLContinuum:Compact:PerEquiv}
Let $X\subset \bbR^d$ be open and bounded, and let $u\in L^1(X;\{\pm 1\})$.
Under assumptions (A1), (B1-2), (C1,3) we have
\begin{align*}
\int_{\partial^*\{u=1\}}\hat{\sigma}(\nu_u(x)) \, \dd \cH^{d-1}(x) < +\infty \quad & \Leftrightarrow \quad \cH^{d-1}(\partial^*\{u=1\})<+\infty \\
 & \Rightarrow \quad \int_{\partial^*\{u=1\}} \sigma^{(p)}(x,\nu_u(x)) \rho(x) \, \dd \cH^{d-1}(x) < +\infty.
\end{align*} \vspace{0\baselineskip}
\end{proposition}

\begin{remark}
The above proposition implies that
\[ \cE_0(u) < +\infty \quad \Leftrightarrow \quad u\in BV(X;\{\pm 1\}) \quad \Leftrightarrow \quad \cG_\infty^{(p)}(u)< +\infty \]
since by definition of $\cG_\infty^{(p)}$ if $\cG_\infty^{(p)}(u)<+\infty$ then $u\in BV(X;\{\pm 1\})$. \vspace{0.8\baselineskip}
\end{remark}

\begin{remark}
The missing implication,
\[ \int_{\partial^*\{u=1\}} \sigma^{(p)}(x,\nu_u(x)) \rho(x) \, \dd \cH^{d-1}(x) < +\infty \implies \cH^{d-1}(\partial^*\{u=1\})<+\infty \,, \]
is also true.
One can show that the minimisation problem in $\sigma^{(p)}(x,\nu)$ (for any $x\in X$ and $\nu\in \mathbb{S}^{d-1}$) can be be reduced to a minimisation problem over functions $f:\bbR\to \bbR$ similar to $\hat{\sigma}(\nu)$ (but with an additional $x$ dependence); see Lemma~\ref{lem:ConvNLContinuum:Preliminary:Cell}.
Since the missing implication is not needed then we do not include it.
\vspace{0.8\baselineskip}
\end{remark}

\begin{proof}[Proof of Proposition~\ref{prop:ConvNLContinuum:Compact:PerEquiv}]
\emph{Step 1.} We show
\[ \cH^{d-1}(\partial^*\{u=1\})<+\infty \quad \implies \quad \int_{\partial^*\{u=1\}}\hat{\sigma}(\nu_u(x)) \, \dd \cH^{d-1}(x) < +\infty. \]
Choose $\hat{f}(t) = +1 $ if $t\geq 0$ and $\hat{f}(t) = -1$ if $t<0$.
Then,
\[
E(\hat{f};\nu) = \int_{-\infty}^\infty \int_{\bbR^d} \eta(h) |\hat{f}(t+h_\nu) - \hat{f}(t) |^2 \, \dd h \, \dd t = \int_{\bbR^d} \eta(h) |h_\nu| \, \dd h \leq R_\eta \int_{\bbR^d} \eta(h) =: C
\]
where $C\in (0,\infty)$, $h_\nu = h\cdot \nu$ and $R_\eta$ is given by assumption (C3).
Note that $C$ is independent of $\nu$.
So,
\[
\int_{\partial^*\{u=1\}}\hat{\sigma}(\nu_u(x)) \, \dd \cH^{d-1}(x) \leq \int_{\partial^*\{u=1\}} E(\hat{f};\nu_u(x)) \, \dd \cH^{d-1}(x) \leq C \cH^{d-1}(\partial^*\{u=1\}).
\]

\emph{Step 2.} We show
\[ \int_{\partial^*\{u=1\}}\hat{\sigma}(\nu_u(x)) \, \dd \cH^{d-1}(x) < +\infty \quad \implies \quad \cH^{d-1}(\partial^*\{u=1\})<+\infty. \]
Using assumption (C1) it is possible to find $a>0$ and $c>0$ satisfying $\eta(h)\geq c$ for all $|h|\leq a$.
Let $\hat{f}:\bbR\to \bbR$ such that
\begin{equation} \label{eq:ConvNLContinuum:Compact:fAss}
\hat{f} \text{ is non-decreasing,} \quad \lim_{t\to\infty} \hat{f}(t) = -\lim_{t\to-\infty} \hat{f}(t) = 1, \quad \hat{f}(t)t\geq 0 \text{ for } t\in\bbR\,.
\end{equation}
If $\hat{f}(\frac{a}{2})\leq \frac{1}{2}$ then $\hat{f}(t) \in  [0,\frac12)$ for $t\in [0,\frac{a}{2}]$ and
\begin{equation}\label{eq:estimateE1}
E(\hat{f};\nu) \geq \frac{a}{2} \inf_{t\in [0,\frac12]} V(t) =: \tilde{a}_1 >0.
\end{equation}
Otherwise, if $\hat{f}\left(\frac{a}{2}\right)> \frac12$ then for $h\geq \frac{a}{2}$
\[
\int_{\bbR} |\hat{f}(t+h)-\hat{f}(t)|^2 \, \dd t \geq \int_{\frac{a}{2}-h}^0 |\hat{f}(t+h) - \hat{f}(t)|^2 \, \dd t \geq \frac{h-\frac{a}{2}}{4}.
\]
Similarly, if $h\leq -\frac{a}{2}$ we have
\[
\int_{\bbR} |\hat{f}(t+h)-\hat{f}(t)|^2 \, \dd t \geq \frac{\frac{a}{2}-h}{4}\,.
\]
So, for all $|h|\geq\frac{a}{2}$ it holds
\[
\int_{\bbR} |\hat{f}(t+h)-\hat{f}(t)|^2 \, \dd t \geq \frac{|h-\frac{a}{2}|}{4}\,.
\]
Therefore
\begin{align}\label{eq:estimateE2}
E(\hat{f};\nu) & \geq \int_{\bbR^d} \eta(h) \int_{\bbR} |\hat{f}(t+h_\nu) - \hat{f}(t)|^2 \, \dd t \, \dd h \nonumber\\
 & \geq \frac{1}{4} \int_{\{h\in\bbR^d\,:\,|h_\nu|\geq \frac{3a}{4}\}} \eta(h) \la h_\nu-\frac{a}{2} \ra \, \dd h \nonumber\\
 & \geq \frac{ac}{16} \int_{\{h\in\bbR^d\,:\,|h_\nu|\geq \frac{3a}{4}\}} \chi_{B(0,a)} \, \dd h \nonumber\\
&=:\tilde{a}_2 >0\,.
\end{align}
Set $\tilde{a} := \min\{\tilde{a}_1,\tilde{a}_2\}>0$. Using \eqref{eq:estimateE1} and \eqref{eq:estimateE2} we have, for any $f:\bbR\to \bbR$ satisfying~\eqref{eq:ConvNLContinuum:Compact:fAss}, that $E(\hat{f};\nu)\geq \tilde{a}$.
We get
\[
\tilde{a} \cH^{d-1}(\partial^*\{u=1\})\leq \int_{\partial^*\{u=1\}} \hat{\sigma}(\nu_u(x)) \, \dd \cH^{d-1}(x)=\cE_0(u)<\infty\,.
\]

\emph{Step 3.} We show
\[ \cH^{d-1}(\partial^*\{u=1\})<+\infty \quad \implies \quad \int_{\partial^*\{u=1\}} \sigma^{(p)}(x,\nu_u(x)) \rho(x) \, \dd \cH^{d-1}(x) < +\infty. \]
Following the same argument as in the first part of the proof we let $\hat{f}(t) = 1$ for $t\geq 0$ and $\hat{f}(t) = -1$ for $t<0$.
Fix $\nu\in\mathbb{S}^{d-1}$ and let $f_\nu(x) = \hat{f}(x\cdot \nu)$.
For any $\bar{x}\in X$ and $C\in \cC(\bar{x},\nu)$ we clearly have $f_\nu\in \cU(C,\nu)$.
So,
\begin{align*}
\frac{1}{\cH^{d-1}(C)} G^{(p)}(f_\nu,\rho(\bar{x}),T_C) & = \rho(\bar{x})\int_{-\infty}^\infty \int_{\bbR^d} \eta(h) |\hat{f}(t+h_\nu) - \hat{f}(t)|^p \, \dd h \, \dd t \\
 & \leq 2^{p}c_2 \int_{\bbR^d}\eta(h) |h_\nu| \, \dd h \\
 & \leq \tilde{c}
\end{align*}
where $\tilde{c}\in (0,\infty)$ can be chosen to be independent of $\nu$ and in the last step we used assumption (C3).
Then,
\begin{align*}
\int_{\partial^*\{u=1\}} \sigma^{(p)}(x,\nu_u(x)) \rho(x) \, \dd \cH^{d-1}(x) & \leq \int_{\partial^*\{u=1\}} \frac{1}{\cH^{d-1}(C)} G^{(p)}(f_\nu,\rho(\bar{x}),T_C) \rho(\bar{x}) \, \dd \cH^{d-1}(\bar{x}) \\
 & \leq c_2 \tilde{c} \cH^{d-1}(\partial^*\{u=1\}).
\end{align*}
This concludes the proof.
\end{proof}

By the above proposition it is enough to show that any sequence $\{u_n\}_{n=1}^\infty\subset L^1(X,\mu)$ satisfying $\sup_{n\in \bbN}\cF_{\eps_n}(u_n)<+\infty$ is relatively compact and any cluster point $u$ satisfies $\cE_0(u)<\infty$.
We do this by a direct comparison with $\cE_{\eps_n}$.

\begin{proof}[Proof of Theorem~\ref{thm:ConvNLContinuum:ConvNLContinuum} (Compactness)]
Assume $\eps_n\to 0^+$, $s_n := s_{\eps_n} \to 1$ and $\sup_{n\in \bbN}\cF^{(p)}_{\eps_n}(u_n)<+\infty$.
Let
\[
v_n(x) :=\mathrm{sign}(u_n):= \lb \begin{array}{ll} +1 & \text{if } u_n(x) \geq 0 \\ -1 & \text{if } u_n(x)<0. \end{array} \rd
\]
We claim that
\begin{equation} \label{eq:ConvNLContinuum:Compact:unvn} 
\|u_n-v_n\|_{L^1}\to 0,
\end{equation}
\begin{equation} \label{eq:ConvNLContinuum:Compact:Fepsvn} 
\sup_{n\in \bbN} \cF_{\eps_n}^{(p)}(v_n)<+\infty.
\end{equation}

\emph{Step 1.} Let us first prove \eqref{eq:ConvNLContinuum:Compact:unvn}.
Fix $\delta>0$ and let
\begin{align*}
K_n^{(\delta)} & = \lb x\in X \, : \, |u_n(x)| \geq 1+\delta \rb \\
L_n^{(\delta)} & = \lb x\in X \, : \, |u_n(x)|\leq 1-\delta \rb.
\end{align*}
Note that for $x\in X\setminus (K_n^{(\delta)}\cup L_n^{(\delta)})$ we have $|v_n(x)-u_n(x)|\leq \delta$.
Now,
\begin{align*}
\int_X |u_n(x) - v_n(x)| \, \dd x & \leq \delta \Vol(X) + \int_{K_n^{(\delta)}} |u_n(x)-v_n(x)| \, \dd x + \int_{L_n^{(\delta)}} |u_n(x) - v_n(x)| \, \dd x \\
 & \leq \delta \Vol(X) + \int_{K_n^{(\delta)}} |u_n(x)| \, \dd x + \Vol(K_n^{(\delta)}) + 2\Vol(L_n^{(\delta)}).
\end{align*}
Since $V$ is continuous and zero only at $\pm 1$ then there exists $\gamma_\delta>0$ such that $V(t)\geq \gamma_\delta$ for all $t\in (-\infty,-1-\delta)\cup (-1+\delta,1-\delta) \cup (1+\delta,+\infty)$.
Hence $V(u_n(x))\geq \gamma_\delta$ for all $x\in K_n^{(\delta)}\cup L_n^{(\delta)}$.
This implies, 
\[ \Vol(K_n^{(\delta)}) \leq \frac{1}{\gamma_\delta} \int_{K_n^{(\delta)}} V(u_n(x)) \, \dd x \leq \frac{\eps_n}{c_1 \gamma_\delta} \cF_{\eps_n}^{(p)}(u_n). \]
By the same calculation $\Vol(L_n^{(\delta)}) \leq \frac{\eps_n}{c_1\gamma_\delta} \cF_{\eps_n}^{(p)}(u_n)$.
Furthermore,
\begin{align*}
\int_{K_n^{(\delta)}} |u_n(x)| \, \dd x & = \int_{K_n^{(\delta)} \cap \{ |u_n(x)|\leq R_V\}} |u_n(x)| \, \dd x + \int_{\{|u_n(x)|>R_V\}} |u_n(x)| \, \dd x \\
 & \leq R_V \Vol(K_n^{(\delta)}) + \frac{1}{\tau} \int_{\{|u_n(x)|>R_V\}} V(u_n(x)) \, \dd x \\
 & \leq \l\frac{R_V}{\gamma_\delta} + \frac{1}{\tau} \r \frac{\eps_n \cF_{\eps_n}^{(p)}(u_n)}{c_1}.
\end{align*}
So, $\lim_{n\to \infty} \| u_n - v_n\|_{L^1} \leq \delta \Vol(X)$.
Since this is true for all $\delta>0$ then we have $\lim_{n\to \infty} \| u_n - v_n\|_{L^1}=0$ which proves claim \eqref{eq:ConvNLContinuum:Compact:unvn}. \vspace{0.5\baselineskip}

\emph{Step 2.} In order to prove \eqref{eq:ConvNLContinuum:Compact:Fepsvn} we reason as follows.
If $|u_n(x)|\geq \frac12$ then $\mathrm{sign}(u_n(x)) \neq \mathrm{sign}(u_n(y))$ implies $|u_n(x)-u_n(y)|\geq \frac12$.
Now since,
\[ |v_n(x) - v_n(y)| = \lb \begin{array}{ll} 0 & \text{if } \mathrm{sign}(u_n(x)) = \mathrm{sign}(u_n(y)) \\ 2 & \text{if } \mathrm{sign}(u_n(x)) \neq \mathrm{sign}(u_n(y)) \end{array} \rd \]
then $|v_n(x) - v_n(y)|\leq 4 |u_n(x) - u_n(y)|$ when $|u_n(x)|\geq \frac12$.

Let $M_n = \{x\in X \, : \, |u_n(x)|\leq \frac12 \}$.
We have,
\begin{align*}
\cF_{\eps_n}^{(p)}(v_n) & = \frac{s_n}{\eps_n} \int_{M_n\times X} \eta_{\eps_n}(x-z) |v_n(x)-v_n(z)|^p \rho(x) \rho(z) \, \dd x \, \dd z \\
 & \quad \quad \quad \quad + \frac{s_n}{\eps_n} \int_{M_n^c\times X} \eta_{\eps_n}(x-z) |v_n(x) - v_n(y)|^p \rho(x) \rho(y) \, \dd x \, \dd z \\
 & \leq \frac{2^{p}c_2^2 s_n}{\eps_n} \int_{M_n\times X} \eta_{\eps_n}(x-z) \, \dd x \, \dd z \\
 & \quad \quad \quad \quad + \frac{4^p s_n}{\eps_n} \int_{M_n^c\times X} \eta_{\eps_n}(x-z) |u_n(x)-u_n(z)|^p \rho(x) \rho(z) \, \dd x \, \dd z \\
 & \leq \frac{2^{p}c_2^2 s_n}{\eps_n} \int_{\bbR^d} \eta(w) \, \dd w \Vol(M_n) + 4^p \cF_{\eps_n}^{(p)}(u_n).
\end{align*}
Now,
\begin{align*}
\frac{1}{\eps_n} \Vol(M_n) \leq \frac{1}{\gamma\eps_n} \int_{\{|u_n(x)|\leq \frac12\}} V(u_n(x)) \, \dd x \leq \frac{1}{\gamma c_1} \cF_{\eps_n}^{(p)}(u_n)
\end{align*}
where $V(t)\geq \gamma>0$ for all $t\in [-\frac12,\frac12]$.
Hence $\sup_{n\in\bbN} \cF_{\eps_n}^{(p)}(v_n)<+\infty$. \vspace{0.5\baselineskip}

\emph{Step 3.} We conclude the proof by noticing that, since $v_n\in L^1(X;\{\pm 1\})$ we have $\cF_{\eps_n}^{(p)}(v_n) = 2^{p-2} s_n \cE_{\eps_n}(v_n)$.
So $v_n$ is relatively compact in $L^1$ by Theorem~\ref{thm:ConvNLContinuum:AB:AB} and by Proposition~\ref{prop:ConvNLContinuum:Compact:PerEquiv} $\cG_\infty^{(p)}(u)<+\infty$ for any cluster point $u$ of $\{v_n\}_{n=1}^\infty$. Using \eqref{eq:ConvNLContinuum:Compact:unvn} the same holds true for the sequence $\{u_n\}_{n\in\bbN}$.
\end{proof}


\subsection{Preliminary results \label{subsec:ConvNLContinuum:Preliminary}}

Here we prove some technical results needed in the proof of the $\Gamma$-convergence result stated in Theorem \ref{thm:ConvNLContinuum:ConvNLContinuum}.
We start by stating the result that the minimisation in the cell problem $\sigma^{(p)}(x,\nu)$ can be restricted to functions $u$ that only vary in the direction $\nu$.

\begin{lemma}
\label{lem:ConvNLContinuum:Preliminary:Cell}
Fix $\nu\in \bbS^{d-1}$ and let $C\subset \nu^\perp$.
Let $p\geq 1$ and assume (B1-3) and (C1-3) are in force.
Then, for any $\lambda>0$, the minimisation problem
\[ \min \lb \frac{1}{\cH^{d-1}(C)} G^{(p)}(u,\lambda,T_C) \,:\, u\in \cU(C,\nu) \rb \]
has a solution $u^\dagger$ that can be written $u^\dagger(x) = \bar{u}(x\cdot \nu)$ where $\bar{u}:\bbR\to [-1,1]$ is increasing.
\end{lemma}

The proof of the lemma is a simple adaptation of~\cite[Theorem 3.3]{alberti98a}.
Indeed, one can absorb $\lambda$ into the mollifier $\eta$ and then the only difference between Lemma~\ref{lem:ConvNLContinuum:Preliminary:Cell} and~\cite[Theorem 3.3]{alberti98a} is the exponent $p$.
By modifying~\cite[Proposition~3.7]{alberti98a} to treat this more general case we can conclude Lemma~\ref{lem:ConvNLContinuum:Preliminary:Cell} holds.
We omit the proof of Lemma~\ref{lem:ConvNLContinuum:Preliminary:Cell}.

\begin{remark}
\label{rem:ConvNLContinuum:Preliminary:Cell}
By Lemma~\ref{lem:ConvNLContinuum:Preliminary:Cell} we can write
\[ \sigma^{(p)}(x,\nu) = \inf \lb \rho(x) \int_{\bbR} \int_{\bbR^d} \eta(h) |u(z + h\cdot \nu) - u(z) |^p \, \dd h \, \dd z + \int_{\bbR} V(u(z)) \, \dd z \rb \]
where the infimum is taken over functions $u:\bbR\to [-1,1]$ satisfying $\lim_{x\to \infty} u(x)=1$, $\lim_{x\to -\infty} u(x)=-1$.
\end{remark}

We continue by proving some continuity properties of the function $\sigma^{(p)}$.

\begin{lemma}\label{lem:ConNLContinuum:Liminf:contsigma}
Under assumptions (A1) and (C3) the followings hold:
\begin{itemize}
\item[(i)] the function $(x,\nu)\mapsto\sigma^{(p)}(x,\nu)$ is upper semi-continuous on $X\times\mathbb{S}^{d-1}$,
\item[(ii)] for every $\nu\in\mathbb{S}^{d-1}$, the function $x\mapsto\sigma^{(p)}(x,\nu)$ is continuous on $X$.
\end{itemize}
\end{lemma}

\begin{proof}
\emph{Proof of (i).} Fix $\bar{x}\in X$ and $\nu\in\mathbb{S}^{d-1}$ and let $\{x_n\}_{n=1}^\infty\subset X$ and $\{\nu_n\}_{n=1}^\infty\subset\mathbb{S}^{d-1}$ be such that $x_n\to\bar{x}$ and $\nu_n\to\nu$ as $n\to\infty$.
Let $R_n$ be a rotation such that $R_n\nu=\nu_n$.
Fix $t>0$ and let $D\in\cC(\bar{x},\nu)$ and $w\in\cU(D,\nu)$ be such that
\[
\frac{1}{\cH^{d-1}(D)}G^{(p)}\left( w,\rho(\bar{x}), T_{D} \right)\leq \sigma^{(p)}(\bar{x},\nu)+t\,.
\]
By Remark~\ref{rem:ConvNLContinuum:Preliminary:Cell} we can assume that $w(x) = \bar{w}(x\cdot \nu)$ for some increasing function $\bar{w}:\bbR\to[-1,1]$.
Then,
\begin{align*}
\int_{T_D} |w(z+h) - w(z)|^p \, \dd z & \leq 2^{p-1} \cH^{d-1}(D) \int_{\bbR} | \bar{w}(z_\nu + h_\nu) - \bar{w}(z_\nu) | \, \dd z_\nu \quad \text{where } h_\nu = h\cdot \nu \\
 & = 2^{p-1} \cH^{d-1}(D) \lim_{M\to \infty} \sum_{i=-M}^{M} \int_0^{h_\nu} |\bar{w}(y+(i+1)h_\nu) - \bar{w}(y+ih_\nu)| \, \dd y \\
 & = 2^{p-1} \cH^{d-1}(D) \lim_{M\to \infty} \int_0^{h_\nu} \bar{w}(y+(M+1)h_\nu) - \bar{w}(y-Mh_\nu) \, \dd y \\
 & = 2^p \cH^{d-1}(D) h_\nu
\end{align*}
with the last line following from the monotone convergence theorem.
Hence,
\begin{equation}\label{eq:inequality}
\int_{T_{D}} |w(z+h) - w(z)|^p \, \dd z \leq  2^p \cH^{d-1}(D) h\cdot \nu\,.
\end{equation}

For $n\in\bbN$ define $C_n\in\cC(x_n,\nu_n)$ and $u_n\in\cU(C_n,\nu_n)$ by
\[
C_n:=R_n(D-\bar{x})+x_n\,,
\]
\[
u_n(x):=w(R_n^{-1}(x-x_n)+\bar{x})\,,
\]
respectively. Then
\begin{align*}
\sigma^{(p)}(x_n,\nu_n)&\leq \frac{1}{\cH^{d-1}(C_n)}G^{(p)}(u_n,\rho(x_n),T_{C_n})\\
&\leq \frac{1}{\cH^{d-1}(D)}G^{(p)}(w,\rho(\bar{x}),T_{D}) + \delta_n\\
&\leq \sigma^{(p)}(\bar{x},\nu)+t+\delta_n\,,
\end{align*}
where
\[
\delta_n:=\frac{1}{\cH^{d-1}(D)} \left|\, G^{(p)}(u_n,\rho(x_n),T_{C_n})-G^{(p)}(w,\rho(\bar{x}),T_{D}) \,\right|\,.
\]
We claim that $\delta_n\to0$ as $n\to\infty$. Since $t>0$ is arbitrary, this will prove the upper semi-continuity.

Fix $\eps>0$.
Using the continuity of $\rho$, for $n$ sufficiently large we have that $|\rho(x_n)-\rho(\bar{x})|<\eps$.
Thus, we get
\begin{align*}
\delta_n & = \frac{1}{\cH^{d-1}(D)} \Bigg| \rho(x_n) \int_{T_{D}}\int_{\bbR^d} \eta(R_n h) \la w(z+h)-w(z) \ra^p \, \dd h \, \dd z \\
 &\hspace{2.5cm}- \rho(\bar{x}) \int_{T_{D}}\int_{\bbR^d} \eta(h) \la w(z+h)-w(z) \ra^p \, \dd h \, \dd z \Bigg| \\
 & \leq \frac{|\rho(x_n)-\rho(\bar{x})|}{\cH^{d-1}(D)} \int_{\bbR^d} \eta(h) \int_{T_{D}}  |w(z+h)-w(z)|^p \, \dd z \, \dd h \\
 &\hspace{2.5cm}+ \frac{\rho(x_n)}{\cH^{d-1}(D)} \int_{\bbR^d} \la \eta(R_nh) - \eta(h) \ra \int_{T_{D}} |w(z+h)-w(z)|^p \, \dd z \, \dd h \\
 & \leq 2^p\eps R_\eta \|\eta\|_{L^1} + 2^p c_2 R_\eta \int_{\bbR^d} \la \eta(R_n h) - \eta(h) \ra \, \dd h 
\end{align*}
where in the last step we used \eqref{eq:inequality}.
In order to show that the second term in the parenthesis vanishes, we use an argument similar to the one for proving that translations are continuous in $L^p$. For every $s>0$ let $\widetilde{\eta}_s:\bbR\to[0,\infty)$ be a continuous function with support in $B(0,2R_\eta)$ such that $\|\eta-\widetilde{\eta}_s\|_{L^1(\bbR^d)}<s$. Then, for every $r>0$ there exists $\bar{n}\in\bbN$ such that $|\widetilde{\eta}_s(R_n h)-\widetilde{\eta}_s(h)|<r$ for all $h\in\bbR^d$ and all $n\geq\bar{n}$.
So that, for $n\geq\bar{n}$
\[
\int_{\bbR^d} \la \eta(R_n h) - \eta(h) \ra \, \dd h\leq \|\eta-\widetilde{\eta}_s\|_{L^1(\bbR^d)}
    +\int_{\bbR^d} |\widetilde{\eta}_s(R_n h)-\widetilde{\eta}_s(h)| \, \dd h \leq s+r\Vol(B(0,2R_\eta))\,.
\]
Since $r$ and $s$ are arbitrary, we conclude that $\int_{\bbR^d} \la \eta(R_n h) - \eta(h) \ra \, \dd h\to0$ as $n\to\infty$ and, in turn, that $\delta_n\to0$ as $n\to\infty$.\vspace{0.5\baselineskip}

\emph{Proof of (ii).} 
Fix $\nu\in \mathbb{S}^{d-1}$, $\bar{x}\in X$ and let $x_n\rightarrow \bar{x}$. 
By part (i) we have that $\sigma^{(p)}(\bar{x},\nu)\geq\limsup_{n\to\infty}\sigma^{(p)}(x_n,\nu)$.
\vspace{0.5\baselineskip}

We claim that $\sigma^{(p)}(\bar{x},\nu)\leq\liminf_{n\to\infty}\sigma^{(p)}(x_n,\nu)$.
Without loss of generality, let us assume that 
\[ \liminf_{n\to\infty}\sigma^{(p)}(x_n,\nu)=\lim_{n\to\infty}\sigma^{(p)}(x_n,\nu)<\infty\,. \]
For every $n\in\bbN$ let $C_n\in\cC(x_n,\nu)$ and $u_n\in\cU(C_n,\nu)$
be such that
\begin{equation}\label{eq:ConNLContinuum:Liminf:est1}
\frac{1}{\cH^{d-1}(C_n)}G^{(p)}(u_n,\rho(x_n),T_{C_n})\leq \sigma^{(p)}(x_n,\nu)+\frac{1}{n}.
\end{equation}
Set $\lambda_n:=\bar{x}-x_n$ and define
\[ \widetilde{C}_n:=C_n+\lambda_n\,,\quad\quad\quad \widetilde{u}_n(x):=u_n(x-\lambda_n). \]
So, $\widetilde{C}_n\in\cC(\bar{x},\nu)$ and $\widetilde{u}_n\in\cU(\widetilde{C}_n,\nu)$.
Using \eqref{eq:ConNLContinuum:Liminf:est1}, we get
\begin{align}\label{eq:estimateliminfsigma}
\sigma^{(p)}(\bar{x},\nu) &\leq \frac{1}{\cH^{d-1}(\widetilde{C}_n)}G^{(p)}(\widetilde{u}_n,\rho(\bar{x}),T_{\widetilde{C}_n}) \nonumber\\
&\leq \frac{G^{(p)}(u_n,\rho(x_n),T_{C_n}) }{\cH^{d-1}(C_n)}
    + \frac{1}{\cH^{d-1}(C_n)}\left|\,G^{(p)}(\widetilde{u}_n,\rho(\bar{x}),T_{\widetilde{C}_n}) - G^{(p)}(u_n,\rho(x_n),T_{C_n}) \,\right| \nonumber\\
& \leq \sigma^{(p)}(x_n,\nu)+\frac{1}{n} + \frac{1}{\cH^{d-1}(C_n)}\left|\,G^{(p)}(\widetilde{u}_n,\rho(\bar{x}),T_{\widetilde{C}_n}) - G^{(p)}(u_n,\rho(x_n),T_{C_n}) \,\right| 
\end{align}

To estimate the last term, we reason as follows.
First of all, we notice that
\begin{equation}\label{eq:vvanishes}
\int_{T_{C_n}} V\left( u_n(z) \right)\dd z = \int_{T_{\widetilde{C}_n}} V\left( \widetilde{u}_n(z) \right)\dd z\,.
\end{equation}
Fix $\eps>0$.
Using the continuity of $\rho$, there exists $\bar{n}\in\mathbb{N}$, such that
$|\rho(x_n)-\rho(\bar{x})|<\eps$ for all $n\geq\bar{n}$.
From \eqref{eq:vvanishes} we get that
\begin{align*}
&\frac{1}{\cH^{d-1}(C_n)}\left|\,G^{(p)}(\widetilde{u}_n,\rho(\bar{x}),T_{\widetilde{C}_n}) - G^{(p)}(u_n,\rho(x_n),T_{C_n}) \,\right|\\
&\quad \quad\quad\quad\leq \frac{1}{\cH^{d-1}(C_n)}|\rho(x_n)-\rho(\bar{x})|
    \int_{T_{C_n}} \int_{\bbR^d} \eta(h) \la u_n(z+h) - u_n(z) \ra^p  \, \dd h \, \dd z  \\
&\quad \quad\quad\quad \leq \frac{\eps}{c_1} \l \sigma^{(p)}(x_n,\nu) + \frac{1}{n} \r \,. \\
\end{align*}
By the above and~\eqref{eq:estimateliminfsigma} we have
\[ \l 1-\frac{\eps}{c_1 n} \r \sigma^{(p)}(\bar{x},\nu) \leq \l 1+\frac{\eps}{c_1}\r \sigma^{(p)}(x_n,\nu) + \frac{1}{n}. \]
Using the arbitrariness of $\eps$ we conclude that
\[
\sigma^{(p)}(\bar{x},\nu)\leq\liminf_{n\to\infty}\sigma^{(p)}(x_n,\nu)\,
\]
as required.
\end{proof}

\begin{remark}
Notice that the above result did not require the existence of a solution for the infimum problem defining $\sigma^{(p)}$. \vspace{0.8\baselineskip}
\end{remark}

We notice that the main feature of the $\Gamma$-convergence of $\cF^{(p)}_{\eps_n}$ to $\cG^{(p)}_\infty$ is that we recover, in the limit, a local functional starting from a nonlocal functional.
To be more precise, let $A,B\subset X$ be disjoint sets. Then, it holds that
\begin{equation}\label{eq:nonloc}
\cF^{(p)}_{\eps_n}(u, A\cup B)=\cF^{(p)}_{\eps_n}(u, A)+\cF^{(p)}_{\eps_n}(u, B)+2\widetilde{\Lambda}_{\eps_n}(u, A, B)
\end{equation}
where we define the nonlocal deficit
\begin{equation} \label{eq:ConvNLContinuum:Prelim:Lambdatilde}
\widetilde{\Lambda}_{\eps}(u, A, B):=\frac{s_{\eps}}{\eps}\int_A\int_B \eta_\eps(x-z) |u(x)-u(z)|^p \rho(x)\rho(z) \, \dd x \, \dd z\,.
\end{equation}
On the other hand, for the limiting functional we have
\begin{equation}\label{eq:local}
\cG^{(p)}_\infty(u, A\cup B)=\cG^{(p)}_\infty(u, A)+\cG^{(p)}_\infty(u, B)\,,
\end{equation}
where, for $u\in BV(X;\{\pm1\})$, we set
\[
\cG^{(p)}_\infty(u, A):=\int_{\partial^*\{u=1\}\cap A} \sigma^{(p)}(x,\nu_u(x)) \rho(x) \, \dd \cH^{d-1}(x)\,.
\]
Identity \eqref{eq:nonloc} states that the functionals $\cF^{(p)}_{\eps_n}$ are nonlocal, while \eqref{eq:local} is the locality property of the limiting functional $\cG^{(p)}_\infty$. Thus, we expect the nonlocal deficit to disappear in the limit,
\emph{i.e.}, that if $u_{\eps_n}\to u$ in $L^1(X)$, then
\begin{equation}\label{eq:nonlocdefvanish}
\widetilde{\Lambda}_{\eps_n}(u_{\eps_n}, A, B)\to0\,,
\end{equation}
as $n\to\infty$. 
For technical reasons we also consider the nonlocal deficits without weighting by $\rho$ or $s_\eps$:
\[
\Lambda_\eps(u, A, B):=\frac{1}{\eps}\int_A\int_B \eta_\eps(x-z) |u(x)-u(z)|^p \, \dd x \, \dd z\,.
\]
By continuity of $\rho$ if $A$ and $B$ are sets in $X$ that are close to $\bar{x}$ then $\widetilde{\Lambda}_\eps(u,A,B)\approx s_\eps\rho^2(\bar{x})\Lambda_\eps(u,A,B)$. 
In~\cite{alberti98} the authors prove that the limit of the nonlocal deficit is determined by the behavior of $u_{\eps_n}$ close to the boundaries of $A$ and $B$ and, in turn, that \eqref{eq:nonlocdefvanish} holds in certain cases of interest. Here we only state the main technical result of \cite{alberti98} in a version we need in the paper, addressing the interested reader to the paper by Alberti and Bellettini for the details.

\begin{proposition}\label{prop:nonlocdef}
Assume (A1) and (C1-3) hold.
Let $v_n\to v$ in $L^1(X)$ with $|v_n|\leq1$. Then, for all $\bar{x}\in\bbR^d$ and for all $\nu\in\mathbb{S}^{d-1}$ the following holds:
given $C\in\mathcal{C}(\bar{x},\nu)$ consider the strip $T_C$ and any cube $Q\subset\bbR^d$ whose intersection with $\nu^\perp$ is $C$, for a.e. $t>0$:
\begin{itemize}
\item[(i)] $\Lambda_{\eps_n}(v_n,tT_C,\bbR^d\setminus tT_C)\to0$ as $n\to\infty$,
\item[(ii)] $\Lambda_{\eps_n}(v_n,tQ,tT_C\setminus tQ)\to0$ as $n\to\infty$.
\end{itemize} \vspace{0\baselineskip}
\end{proposition}

\begin{remark}\label{rem:nonlocldef}
The boundness assumption on the sequence $\{v_n\}_{n=1}^\infty$ allows one to obtain the proof of the above result directly from~\cite[Proposition 2.5 and Theorem 2.8]{alberti98}.
In particular, in~\cite{alberti98} the authors prove the result for $p=1$ and $\rho\equiv 1$, using the $L^\infty$ bound on $v_n$ one can easily bound the more general case considered here by the $L^1$ case. 
With similar computations it is also possible to obtain the same result without the $L^\infty$ bound. 

Finally, notice that when $A,B\in\bbR^d$ are disjoint sets with $\mathrm{d}(A,B)>0$, using the fact that the function $\eta$ has support in the ball $B(0,R_\eta)$ (see (C3)), it is easy to prove that there exists $\bar{n}\in\mathbb{N}$ such that for all $n\geq\bar{n}$ it holds
\[
\Lambda_{\eps_n}(v_n, A, B)=0\,.
\]
\vspace{0\baselineskip}
\end{remark}

For technical reasons we need to introduce a scaled version of the functional $G^{(p)}$.

\begin{mydef}
For $\eps>0$, $p\geq1$, $u:\bbR^d\to\bbR$, $\lambda\in\bbR$, and $A\subset\bbR^d$, we define
\[
G^{(p)}_\eps(u,\lambda,A):= \frac{\lambda}{\eps} \int_{A} \int_{\bbR^d} \eta_\eps(h)
    |u(z+h) - u(z)|^p \, \dd h \, \dd z + \frac{1}{\eps}\int_{A} V(u(z)) \, \dd z\,.
\]
\vspace{0\baselineskip}
\end{mydef}

Let $r>0$ and $x\in X$. For a set $A\subset\bbR^d$, we define $x+rA:=\{ x+ry \,:\, y\in A \}$.
Moreover, for a function $u:\bbR^d\to\bbR$, we set
\[ R_{x,r}u(y):=u(x+ry)\,. \]
Using a change of variables, it is easy to see that the following scaling property holds true:
\begin{equation}\label{eq:scalprop1}
G^{(p)}_\eps(u,\lambda, x+rA)=r^{d-1}G^{(p)}_{\eps/r}(R_{x,r}u,\lambda, A)\,.
\end{equation}


\subsection{The Liminf Inequality \label{subsec:ConvNLContinuum:Liminf}}

This section is devoted to proving the following: let $u_{\eps_n}\rightarrow u$ in $L^1$, then
\begin{equation}\label{eq:liminftoprove}
\cG^{(p)}_{\infty}(u)\leq\liminf_{n\rightarrow\infty}\cF^{(p)}_{\eps_n}(u_{\eps_n})\,.
\end{equation}
We will follow the proof of \cite[Theorem 1.4]{alberti98}, with some modifications due to the presence of the density~$\rho$.

\begin{proof}[Proof of Theorem~\ref{thm:ConvNLContinuum:ConvNLContinuum} (Liminf).]
Let $\eps_n\to 0^+$ and $u_{\eps_n}\to u$ in $L^1(X,\mu)$.
Assume without loss of generality that
\begin{equation}\label{eq:bound}
\liminf_{n\rightarrow\infty}\cF^{(p)}_{\eps_n}(u_{\eps_n})
    =\lim_{n\rightarrow\infty}\cF^{(p)}_{\eps_n}(u_{\eps_n})<\infty\,.
\end{equation}

\emph{Step 1.} By compactness (see Section \ref{subsec:ConvNLContinuum:Compact}) it holds $u=\chi_A$ for some set $A\subset X$  of finite perimeter in $X$. In order to prove \eqref{eq:liminftoprove} we use the strategy introduced by Fonseca and M\"{u}ller in \cite{fonsecamuller93}. Write
\begin{equation}\label{eq:writing}
\cF^{(p)}_{\eps_n}(u_{\eps_n})=\int_X g_{\eps_n}(x)\dd x
\end{equation}
and set $\dd \lambda_{\eps_n}:= g_{\eps_n} \dd \mathcal{L}^d\res X
$,
so that
\begin{equation}\label{eq:writing1}
|\lambda_{\eps_n}|(X)=\cF^{(p)}_{\eps_n}(u_{\eps_n})\,.
\end{equation}
Using \eqref{eq:bound}, \eqref{eq:writing}, \eqref{eq:writing1}
, up to a subsequence (not relabeled) it holds $\lambda_{\eps_n}\stackrel{*}{\rightharpoonup}\lambda$
for some finite Radon measure $\lambda$ on $X$.
Then
\begin{equation}\label{eq:lsctv}
|\lambda|(X)\leq\liminf_{n\to\infty}|\lambda_{\eps_n}|(X)\,.
\end{equation}
In view of \eqref{eq:writing1} and \eqref{eq:lsctv}, the liminf inequality \eqref{eq:liminftoprove} is implied by the following claim: for $\cH^{d-1}$-a.e. $\bar{x}\in \partial^*\{u=1\}$ it holds
\[
\sigma^{(p)}(\bar{x},\nu(\bar{x}))\rho(\bar{x})\leq\frac{\mathrm{d}\lambda}{\mathrm{d}\theta}(\bar{x})\,,
\]
where $\theta:=\cH^{d-1}\res \partial^*\{u=1\}$.
In order to prove the claim we reason as follows. For $\cH^{d-1}$-a.e. $\bar{x}\in \partial^*\{u=1\}$ it is possible to find the density of $\lambda$ with respect to $\theta$ via (recall Remark \ref{rem:newdefnormal})
\begin{equation}\label{eq:density}
\frac{\mathrm{d}\lambda}{\mathrm{d}\theta}(\bar{x})=\lim_{r\to0}\,\frac{\lambda(\bar{x}+r Q)}{r^{d-1}}\,,
\end{equation}
where $Q$ is a unit cube centered at the origin and having $\nu(\bar{x})$, the measure theoretic exterior normal to $A$ at $\bar{x}$, as one of its axes. Let $\bar{x}\in \partial^*\{u=1\}$. Theorem \ref{thm:DeGiorgi} implies that
\begin{equation}\label{eq:convhalfplane}
R_{\bar{x},r}u\rightarrow v_{\bar{x}}
\end{equation}
in $L^1_{loc}(\mathbb{R}^N)$ as $r\rightarrow0$, where
\[
v_{\bar{x}}(x):=
\left\{
\begin{array}{ll}
-1 & x\cdot \nu(\bar{x})\geq0\,,\\
1 & x\cdot \nu(\bar{x})<0\,.
\end{array}
\right.
\]
Let $\bar{x}\in \partial^*\{u=1\}$ be a point for which \eqref{eq:density} and \eqref{eq:convhalfplane} hold.
Without loss of generality, we can assume that
\begin{equation}\label{eq:densfinite}
\frac{\mathrm{d}\lambda}{\mathrm{d}\theta}(\bar{x})<\infty\,.
\end{equation}
Since $u_{\eps_n}\rightarrow u$ in $L^1$ and $\lambda_{\eps_n}\stackrel{*}{\rightharpoonup}\lambda$, it is possible to find a (not relabeled) subsequence
$\{\eps_n\}_{n=1}^\infty$ and a sequence $\{r_n\}_{n=1}^\infty$
with $r_n\rightarrow0^+$ and $\frac{\eps_n}{r_n}\rightarrow0^+$, such that
\begin{equation}\label{eq:conv1}
\frac{\dd \lambda}{\dd\theta}(\bar{x})=\lim_{n\rightarrow\infty}\frac{\lambda_{\eps_n}(\bar{x}+r_nQ)}{r_n^{d-1}}
\end{equation}
and
\[
R_{\bar{x},r_n}u_{\eps_n}\rightarrow v_{\bar{x}}\,.
\]
Using the fact that $X$ is open, we can assume that $\bar{x}+r_n Q\subset X$ for all $n\in\mathbb{N}$. Thus
\begin{equation}\label{eq:est2}
\frac{\lambda_{\eps_n}(\bar{x}+r_nQ)}{r_n^{d-1}}\geq\frac{\cF^{(p)}_{\eps_n}(u_{\eps_n},\bar{x}+r_n Q)}{r_n^{d-1}}\,.
\end{equation}
\vspace{0.5\baselineskip}

\emph{Step 2.} We claim that
\begin{equation}\label{eq:deltantozero}
\delta_n:=\frac{|\cF^{(p)}_{\eps_n}(u_{\eps_n},\bar{x}+r_nQ) - 
   \rho(\bar{x})\widetilde{\mathcal{F}}^{(p)}_{\eps_n}(u_{\eps_n},\rho(\bar{x}),\bar{x}+r_nQ) |}{r_n^{d-1}}\rightarrow0\,,
\end{equation}
as $n\to\infty$, where
\[
\widetilde{\mathcal{F}}^{(p)}_{\eps}(u,\xi,A):= \frac{\xi s_\eps}{\eps}\int_A \int_A \eta_{\eps}(x-z)|u(x)-u(z)|^p\,\dd x \, \dd z
    +\frac{1}{\eps}\int_A V(u(x)) \, \dd x\,,
\]
for $\eps>0$, $A\subset X$, $u:A\to\mathbb{R}$ and $\xi\in\mathbb{R}$.
Indeed, fix $t>0$. Thanks to Assumption (A1) the function $\rho$ is continuous in $X$. Then, it is possible to find $\bar{n}\in\mathbb{N}$ such that for all $n\geq\bar{n}$ and all $y\in \bar{x}+r_nQ$ it holds
\[
|\rho(y)-\rho(\bar{x})|<t\,.
\]
Thus,
\begin{align*}
\delta_n&\leq \frac{t}{r_n^{d-1}\eps_n} \int_{\bar{x}+r_n Q} V(u_{\eps_n}(x)) \, \dd x
+\frac{ts_{\eps_n}\rho(\bar{x})}{r_n^{d-1}\eps_n} \int_{\bar{x}+r_n Q} \int_{\bar{x}+r_n Q} \eta_{\eps_n}(x-z)|u_{\eps_n}(x)-u_{\eps_n}(z)|^p\,\dd x \, \dd z\\
&\hspace{0.6cm} + \frac{t s_{\eps_n}}{r_n^{d-1} \eps_n} \int_{\bar{x}+r_n Q} \int_{\bar{x}+r_n Q} \eta_{\eps_n}(x-z)|u_{\eps_n}(x)-u_{\eps_n}(z)|^p \rho(z)\,\dd x \, \dd z\\
&\leq \frac{t}{r_n^{d-1}\eps_n c_1} \int_{\bar{x}+r_n Q} V(u_{\eps_n}(x))\rho(x) \, \dd x \\
& \hspace{0.6cm}   + \frac{t s_{\eps_n}(c_1+c_2)}{r_n^{d-1}\eps_n c_1^2} \int_{\bar{x}+r_n Q} \int_{\bar{x}+r_n Q} \eta_{\eps_n}(x-z)|u_{\eps_n}(x)-u_{\eps_n}(z)|^p \rho(z)\rho(x)\,\dd x \, \dd z \\
&\leq \frac{t(c_1+c_2)}{c_1^2}\frac{\lambda_{\eps_n}(\bar{x}+r_nQ)}{r_n^{d-1}} \,,
\end{align*}
where in the last step we used \eqref{eq:est2}.
By~\eqref{eq:densfinite} and~\eqref{eq:conv1} $\lim_{n\to \infty}\delta_n\leq Ct$ for some constant $C<\infty$.
Since $t>0$ is arbitrary, this proves the claim. \vspace{0.5\baselineskip}

\emph{Step 3.} Observe that for any $\lambda\geq 0$, $\eps>0$, $r>0$ and $v\in L^1$ we have
\[ \min\lb 1,\frac{s_{\eps}}{s_{\frac{\eps}{r}}} \rb \tilde{\cF}_{\frac{\eps}{r}}^{(p)}(R_{\bar{x},r}v,\lambda,Q) \leq \frac{1}{r^{d-1}} \tilde{\cF}_\eps^{(p)}(v,\lambda,\bar{x}+rQ) \leq \max\lb 1,\frac{s_{\eps}}{s_{\frac{\eps}{r}}} \rb \tilde{\cF}_{\frac{\eps}{r}}^{(p)}(R_{\bar{x},r}v,\lambda,Q). \]
Let $C=Q\cap \nu(\bar{x})^\perp \in \cC(\bar{x},\nu(\bar{x}))$.
Define the function $w_n:\bbR^d\to\bbR$ as the periodic extension of the function that is $R_{\bar{x},r_n}\,u_{\eps_n}$ in $Q$ and $v_{\bar{x}}$ in $T_C\setminus Q$.
Set $\eps'_n:=\frac{\eps_n}{r_n}$ and $s_n^\prime=\min\lb 1,\frac{s_{\eps_n}}{s_{\eps'_n}}\rb$.
Using \eqref{eq:est2} and \eqref{eq:deltantozero} together with the scaling identity \eqref{eq:scalprop1} we get
\begin{align}\label{eq:finalest1}
\frac{\lambda_{\eps_n}(\bar{x}+r_nQ)}{r_n^{d-1}}&\geq
    \rho(\bar{x})\frac{\widetilde{\mathcal{F}}_{\eps_n}^{(p)}(u_{\eps_n},\rho(\bar{x}),\bar{x}+r_nQ)}{r_n^{d-1}}-\delta_n
    \nonumber\\
&\geq s^\prime_n \rho(\bar{x})\widetilde{\mathcal{F}}^{(p)}_{\eps'_n}(R_{\bar{x},r_n}\,u_{\eps_n},\rho(\bar{x}),Q)-\delta_n
    \nonumber\\
&=s^\prime_n \rho(\bar{x})\widetilde{\mathcal{F}}^{(p)}_{\eps'_n}(w_n,\rho(\bar{x}),Q)-\delta_n
    \nonumber\\
&\geq s^\prime_n \rho(\bar{x})G^{(p)}_{\eps'_n}(w_n,\rho(\bar{x}),T_C)-\delta_n \nonumber\\
&\hspace{0.6cm}-s^\prime_n\rho(\bar{x})\left|\,\widetilde{\mathcal{F}}^{(p)}_{\eps'_n}(w_n,\rho(\bar{x}),Q) -
    G^{(p)}_{\eps'_n}(w_n,\rho(\bar{x}),T_C)   \,\right| \nonumber\\
&=s^\prime_n\rho(\bar{x})\left(\frac{\eps_n}{r_n}\right)^{d-1}G^{(p)}\left(R_{0,\eps'_n}w_n,\rho(\bar{x}),\frac{r_n}{\eps_n}T_C\right)-\delta_n \nonumber\\
&\hspace{0.6cm}-s^\prime_n\rho(\bar{x})\left|\,\widetilde{\mathcal{F}}^{(p)}_{\eps'_n}(w_n,\rho(\bar{x}),Q) -
    G^{(p)}_{\eps'_n}(w_n,\rho(\bar{x}),T_C)   \,\right| \,.
\end{align}
We would like to say that 
\[
\left|\,\widetilde{\mathcal{F}}^{(p)}_{\eps'_n}(w_n,\rho(\bar{x}),Q) -G^{(p)}_{\eps'_n}(w_n,\rho(\bar{x}),T_C)   \,\right|\to0
\]
as $n\to\infty$. Unfortunately, this might not be true. In order to overcome this difficulty, take $t\in (0,1)$.
Notice that we can bound
\begin{align}\label{eq:finalest2}
& \la\widetilde{\cF}^{(p)}_{\eps'_n}(w_n,\rho(\bar{x}),tQ) - G^{(p)}_{\eps'_n}(w_n,\rho(\bar{x}),tT_C)\ra
\leq 2\rho(\bar{x})\Lambda_{\eps'_n}(w_n,tQ,tT_C\setminus tQ) \nonumber\\
&\quad\quad\quad+ \rho(\bar{x}) \Lambda_{\eps'_n}(w_n,tT_C,\bbR^d\setminus tT_C)
    + \rho(\bar{x}) \Lambda_{\eps'_n}(w_n,tT_C\setminus tQ,tT_C\setminus tQ) \nonumber\\
&\quad\quad\quad+\frac{r_n\rho(\bar{x})}{\eps_n} |1-s_{\eps'_n}| \int_{tQ} \int_{tQ} \eta_{\frac{\eps_n}{r_n}}(y-z) |w_n(y) - w_n(z)|^p \, \dd y \, \dd z \,.
\end{align}
Now,
\begin{align}\label{eq:finalest3}
\frac{r_n\rho(\bar{x})}{\eps_n} \int_{tQ} \int_{tQ} \eta_{\eps'_n}(y-z) |w_n(y) - w_n(z)|^p \, \dd y \, \dd z
    & \leq G_{\eps'_n}^{(p)}(w_n,\rho(\bar{x}),T_C) \\
& = \l\frac{\eps_n}{r_n}\r^{d-1} G^{(p)}\left(R_{0,\eps'_n}w_n,\rho(\bar{x}),\frac{r_n}{\eps_n}T_C\right)\,. \nonumber
\end{align}
Moreover, using Proposition \ref{prop:nonlocdef} we get that for a.e. $t\in(0,1)$ it holds
\begin{equation}\label{eq:finalest4}
\Lambda_{\eps'_n}(w_n,tQ,tT_C\setminus tQ)\to0\,,\quad\quad\quad \Lambda_{\eps'_n}(w_n,tT_C,\bbR^d\setminus tT_C)\to0
\end{equation}
as $n\to\infty$. Finally, using Remark \ref{rem:nonlocldef} and the fact that $w_n$ is constant on $T_C\setminus Q$ it is easy to see that
\begin{equation}\label{eq:finalest5}
\lim_{t\to1}\lim_{n\to\infty}\Lambda_{\eps'_n}(w_n,tT_C\setminus tQ,tT_C\setminus tQ)=0\,.
\end{equation}
Hence, from \eqref{eq:deltantozero}, \eqref{eq:finalest1}, \eqref{eq:finalest2}, \eqref{eq:finalest3}, \eqref{eq:finalest4} and \eqref{eq:finalest5} and recalling that $s'_n\to1$ we get
\[
\lim_{n\to\infty}\frac{\lambda_{\eps_n}(\bar{x}+r_nQ)}{r_n^{d-1}} \geq \rho(\bar{x})\sigma^{(p)}(\bar{x},\nu(\bar{x}))
\]
as required
\end{proof}


\subsection{The Limsup Inequality \label{subsec:ConvNLContinuum:Limsup}}

This section is devoted at proving the following: let $u\in BV(X,\{\pm1\})$, then it is possible to find $\{u_{\eps_n}\}_{n=1}^\infty\subset L^1(X)$ with $u_{\eps_n}\to u$ in $L^1(X,\mu)$ such that
\begin{equation}\label{eq:limsuptoprove}
\limsup_{n\rightarrow\infty}\cF^{(p)}_{\eps_n}(u_{\eps_n})\leq\cG_\infty^{(p)}(u)\,.
\end{equation}
Without loss of generality, we can assume $\cG_\infty^{(p)}(u)<\infty$, namely $u\in BV(X;\{\pm1\})$.
The proof will follow the lines of the argument used to prove \cite[Theorem 5.2]{alberti98}. 

\begin{proof}[Proof of Theorem~\ref{thm:ConvNLContinuum:ConvNLContinuum} (limsup).]
We first prove the result for polyhedral functions then, via a diagonalisation argument, generalise to arbitrary functions in $BV(X;\{\pm1\})$.
We fix the sequence $\eps_n\to 0^+$ now.
\vspace{0.5\baselineskip}

\emph{Step 1. Polyhedral functions.} Assume $u\in BV(X,\{\pm1\})$ is a polyhedral function (see Definition \ref{def:polyfun}).
Then we claim that there exists a sequence $\{u_{\eps_n}\}_{n=1}^\infty$ with $|u_{\eps_n}|\leq1$, converging uniformly to $u$ on every compact set $K\subset X\setminus \partial^*\{u=1\}$, and in particular $u_{\eps_n}\rightarrow u$ in $L^1(X,\mu)$, such that \eqref{eq:limsuptoprove} holds.

Let us denote by $E$ the polyhedral set $\{u=1\}$ and by $E_1,\dots,E_k$ its faces.
It is possible to cover $\partial E\cap X$ with a finite family of sets $\bar{A}_1,\dots,\bar{A}_k$, where each $A_i$ is an open set satisfying the following properties:
\begin{itemize}
\item[(i)] $\partial A_i$ can be written as the union of two Lipschitz graphs over the face $E_i$,
\item[(ii)] every point in the relative interior of $E_i$ belongs to $A_i$,
\item[(iii)] $\cH^{d-1}\left(\bar{A}_i\cap \bigcup_{j\neq i}E_j\right)=0$,
\item[(iv)] $A_i\cap A_j=\emptyset$ if $i\neq j$.
\end{itemize}
Set
\[
A_0:=\{u=1\}\setminus\bigcup_{i=1}^k\bar{A}_i\,,\quad\quad\quad A_{k+1}:=X\setminus\bigcup_{i=0}^k\bar{A_i}\,.
\]

\begin{figure}
\centering
\includegraphics[scale=1]{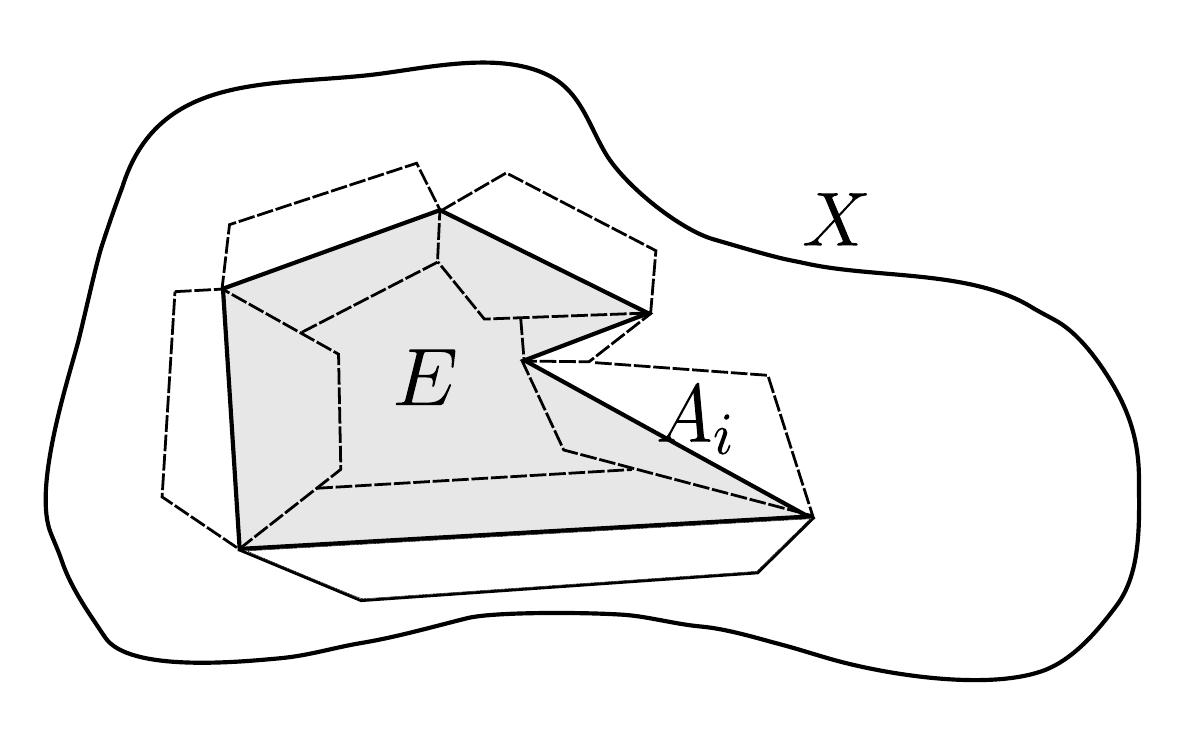}
\caption{The polyhedral set $E$ (shaded) and the sets $A_i$ (dotted lines).}
\end{figure}

We then define $u_{\eps_n}$ in each $A_i$ separately.
Set $u_{\eps_n}(x):=1$ for $x\in A_0$ and $u_{\eps_n}(x):=-1$ for $x\in A_{k+1}$.
Now fix $i\in\{1,\dots,k\}$ and $n\in\bbN$.
In order to define $u_{\eps_n}$ in $A_i$, we reason as follows.
Denote by $\nu$ the normal of the hyperplane containing $E_i$.
Without loss of generality, we can assume $\nu=e_d$ and $E_i\subset\{x_d=0\}$.
A point $x\in\mathbb{R}^d$ will be denoted as
\[
x=(x',x_d)\,,\quad\quad x'\in\mathbb{R}^{d-1}\,,\quad x_d\in\mathbb{R}\,.
\]
Fix $\xi>0$.
Using the continuity of $\sigma$ (see Lemma \ref{lem:ConNLContinuum:Liminf:contsigma}) and of $\rho$ in $X$ (see (A1)) it is possible to find a finite family of $d-1$ dimensional disjoint cubes $\{Q_j\}_{j=1}^{M_\xi}$, for some $M_\xi\in\bbN$, of side $r_\xi>0$ lying in the hyperplane containing $E_i$, having $E_i\subseteq \cup_{j=1}^{M_\xi} Q_j$, $E_i\cap Q_j \neq \emptyset$, and satisfying the following properties: denoting by $\{x_j\}_{j=1}^{M_\xi}$ their centers (or, in the case the center of a cube $Q_j$ is not contained in $E_i$, a point of $Q_j\cap E_i$) we have
\begin{equation}\label{eq:cond1}
\left|\, \int_{E_i} \sigma^{(p)}(x,\nu)\rho(x)\,\dd\cH^{d-1}(x) -
    r^{d-1}_\xi \sum_{j=1}^{M_\xi}  \sigma^{(p)}(x_j,\nu)\rho(x_j) \,\right|<\xi\,.
\end{equation}
It is possible to find, for every $j=1,\dots M_\xi$, $C_j\in\cC(x_j,\nu)$ and $w_j\in \cU(C_j,\nu)$ such that
\begin{equation}\label{eq:condu0}
\frac{1}{\cH^{d-1}(C_j)}G^{(p)}(w_j,\rho(x_j),T_{C_j})<\sigma^{(p)}(x_j,\nu)+\frac{\xi}{\rho(x_j)M_\xi r^{d-1}_\xi}\,.
\end{equation}
We can assume $|w_j|\leq 1$.
For every $j=1,\dots,M_\xi$, let $L_j\in\bbN$ be such that
\[ \frac{1}{\cH^{d-1}(C_j)}\left|\, G^{(p)}(w_j,\rho(x_j),T_{C_j})-G^{(p)}(\widetilde{w}_j,\rho(x_j),T_{C_j})  \,\right|<\frac{\xi}{\rho(x_j)M_\xi r^{d-1}_\xi}\,, \]
where
\[ \widetilde{w}_j(x):=
\left\{
\begin{array}{ll}
w_j(x) & \text{if }  |x_d|<L_j \,,\\
+1 & \text{if } x_d>L_j\,,\\
-1 & \text{if } x_d<-L_j\,.
\end{array}
\right. \]
By mollifying, $\widetilde{w}_j$, i.e. $\widetilde{w}_j^\prime = J_\beta\ast \widetilde{w}_j$, we have that $\widetilde{w}_j^\prime$ is Lipschitz continuous with $\Lip(\widetilde{w}_j^\prime) \leq \frac{1}{\beta}$ and $\widetilde{w}_j^\prime(x)\in \{\pm 1\}$ for $|x_d|>L_j+\beta$.
We can choose $\beta=\beta_\xi$ such that
\begin{equation}\label{eq:condun}
\frac{1}{\cH^{d-1}(C_j)}\left|\, G^{(p)}(w_j,\rho(x_j),T_{C_j})-G^{(p)}(\widetilde{w}_j^\prime,\rho(x_j),T_{C_j})  \,\right|<\frac{2\xi}{\rho(x_j)M_\xi r^{d-1}_\xi}\,.
\end{equation}

For every $j=1,\dots,M_\xi$ cover $Q_j$ with copies of $\eps_n C_j$.
Denote them by $\{Q_{j,s}^{(n)}\}_{s=1}^{k^{(n)}_j}$ and by $\{y^{(n)}_{j,s}\}_{s=1}^{k^{(n)}_j}$ their centers.
Notice that it might be necessary to consider the intersection of some cubes with $Q_j$, and that
\begin{equation}\label{eq:ratek}
k^{(n)}_j=\Bigg\lceil \frac{1}{\eps_n^{d-1}}r^{d-1}_\xi\left(\cH^{d-1}(C_j)\right)^{-1} \Bigg\rceil\,,
\end{equation}
where $\lceil x \rceil$ denotes the smallest natural number greater than or equal to $x\in\bbR$.
Define the function $v_j^{(n)}:\bbR^d\to[-1,+1]$ as the periodic extension of
\[ v^{(n)}_j(x):=\sum_{s=1}^{k_j^{(n)}}\widetilde{w}^\prime_j\left(\frac{x-(x_j+y^{(n)}_j)}{\eps_n}\right)\chi_{Q^{(n)}_{j,s}}(x')\,. \]
We are now in position to define the function $u_{\eps_n}$ in $A_i$: for $x\in A_i$ define
\[
u_{\eps_n}(x):=\sum_{j=1}^{M_\xi} \chi_{Q^{(n)}_j}(x')v^{(n)}_j(x)\,.
\]
Note that $u_{\eps_n}$ has Lipschitz constant $\Lip(u_{\eps_n}) \leq \frac{1}{\beta_\xi \eps_n}$ and $u_{\eps_n}(x) = u(x)$ for any $x$ with $\dist(x,\partial^*\{u=1\}) \geq \max_j (L_j+\beta_\xi) \eps_n$.
Hence, $u_{\eps_n}\rightarrow u$ as $n\to\infty$ uniformly on compact sets $K\subset X\setminus \partial^*\{u=1\}$.
We now prove the validity of inequality \eqref{eq:limsuptoprove}.
We claim that:
\begin{itemize}
\item[(i)] $\widetilde{\Lambda}_{\eps_n}(u_{\eps_n},A_i,A_j)\to0$ as $n\to\infty$ for all $i\neq j$, where $\widetilde{\Lambda}$ is defined by~\eqref{eq:ConvNLContinuum:Prelim:Lambdatilde};
\item[(ii)] $\limsup_{n\to\infty}\cF^{(p)}_{\eps_n}(u_{\eps_n},A_i)\leq \cG^{(p)}_\infty(u,A_i)$ for all $i=0,\dots,k+1$.
\end{itemize}
If the above claims hold true, then we can conclude as follows: we have
\begin{align*}
\limsup_{n\to\infty}\cF^{(p)}_{\eps_n}(u_{\eps_n})&\leq \sum_{i=0}^{k+1} \limsup_{n\to\infty}\cF^{(p)}_{\eps_n}(u_{\eps_n},A_i)
    + 2\sum_{i<j=0}^{k+1} \limsup_{n\to\infty}\widetilde{\Lambda}_{\eps_n}(u_{\eps_n},A_i,A_j)\\
&\leq \sum_{i=0}^{k+1} \cG^{(p)}_\infty(u,A_i)\\
&= \cG^{(p)}_\infty(u)\,.
\end{align*}
We start by proving claim (ii). It is easy to see that it holds true for $i=0, k+1$.
Fix $i\in\{1,\dots,k\}$.
Noticing that
\[
A_i\subset \bigcup_{j=1}^{M}\bigcup_{s=1}^{k_j^{(n)}}T_{Q^{(n)}_{j,s}}\,,
\]
we get
\begin{align}\label{eq:firstpartconmput}
\cF^{(p)}_{\eps_n}(u_{\eps_n},A_i)&\leq \cF^{(p)}_{\eps_n}\left(u_{\eps_n},
    \bigcup_{j=1}^{M_\xi}\bigcup_{s=1}^{k_j^{(n)}} T_{Q^{(n)}_{j,s}}  \right) \nonumber\\
&\leq \sum_{j=1}^{M_\xi}\sum_{s=1}^{k_j^{(n)}}\Biggl[\,
    \frac{s_{\eps_n}}{\eps_n} \int_{T_{Q^{(n)}_{j,s}}}\int_{\bbR^d} \eta_{\eps_n}(x-z)
        |u_{\eps_n}(x)-u_{\eps_n}(z)|^p \rho(x)\rho(z) \, \dd x \, \dd z \nonumber\\
&\hspace{2.6cm} + \frac{1}{\eps_n} \int_{T_{Q^{(n)}_{j,s}}} V(u_{\eps_n}(x)) \rho(x) \, \dd x \,\Biggr]\,.
\end{align}
Here, for every $j$ and $s$ we are using the whole cube $Q^{(n)}_{j,s}$, not only its intersection with $Q_j$.

Note that,
\begin{align*}
\cA^{(n)}_{j,s} := & \lb (x,z) \in \bbR^d\times T_{Q^{(n)}_{j,s}} \, : \, \eta_{\eps_n}(x-z)|u_{\eps_n}(x)-u_{\eps_n}(z)|^p \neq 0 \rb \\
& \hspace{2cm} \subseteq B(x_j,\sqrt{2d} r_\xi) \times \l T_{Q^{(n)}_{j,s}} \cap \lb |z_d|\leq \eps_n(L_j+R_\eta) \rb \r
\end{align*}
for $\eps_n$ sufficiently small (compared to $r_\xi$).
Hence there exists $\delta_\xi>0$ such that for all $(x,z)\in \cA^{(n)}_{j,s}$ we have $|\rho(x)-\rho(x_j)|\leq \delta_\xi$ and $|\rho(z)-\rho(x_j)|\leq \delta_\xi$ where $\delta_\xi\to 0$ as $\xi\to 0^+$.
This implies
\begin{align*}
s_{\eps_n} \rho(x) \rho(z) & \leq |s_{\eps_n} - 1| \rho(x) \rho(z) + \rho(x) \rho(z) \\
 & \leq |s_{\eps_n}-1|c_2^2 + \rho(x_j) \rho(z) + \delta_\xi c_2 \\
 & \leq |s_{\eps_n}-1|c_2^2 + \rho^2(x_j) + 2\delta_\xi c_2\,.
\end{align*}
Hence,
\begin{align*}
& \frac{s_{\eps_n}}{\eps_n} \int_{T_{Q_{j,s}^{(n)}}} \int_{\bbR^d} \eta_{\eps_n}(x-z) |u_{\eps_n}(x) - u_{\eps_n}(z)|^p \rho(x) \rho(z) \, \dd x \, \dd z \\
& \hspace{2cm} \leq \l |s_{\eps_n}-1|c_2^2 + 2\delta_\xi c_2 + \rho^2(x_j) \r \frac{1}{\eps_n} \int_{T_{Q_{j,s}^{(n)}}} \int_{\bbR^d} \eta_{\eps_n}(x-z) |u_{\eps_n}(x) - u_{\eps_n}(z)|^p \, \dd x \, \dd z\,.
\end{align*}
Similarly,
\[ \frac{1}{\eps_n} \int_{T_{Q^{(n)}_{j,s}}} V(u_{\eps_n}(x)) \rho(x) \, \dd x \leq \l \rho(x_j) + \delta_\xi \r \frac{1}{\eps_n} \int_{T_{Q^{(n)}_{j,s}}} V(u_{\eps_n}(x)) \, \dd x\,.  \]

Let $\gamma_{\xi,n} = 1+\frac{2\delta_\xi}{c_2} + |s_{\eps_n}-1|$ and notice that
\[ \lim_{\xi\to 0} \lim_{n\to \infty} \gamma_{\xi,n} = 1 \]
(there is a subtle dependence that requires $\eps_n$ be small compared to $r_\xi$ which means it is necessary to first take the limits in this order).
Then
\begin{align*}
& \frac{s_{\eps_n}}{\eps_n} \int_{T_{Q_{j,s}^{(n)}}} \int_{\bbR^d} \eta_{\eps_n}(x-z) |u_{\eps_n}(x) - u_{\eps_n}(z)|^p \rho(x) \rho(z) \, \dd x \, \dd z + \frac{1}{\eps_n} \int_{T_{Q^{(n)}_{j,s}}} V(u_{\eps_n}(x)) \rho(x) \, \dd x \\
& \,\, \leq \gamma_{\xi,n} \rho(x_j) \l \frac{\rho(x_j)}{\eps_n} \int_{T_{Q_{j,s}^{(n)}}} \int_{\bbR^d} \eta_{\eps_n}(x-z) |u_{\eps_n}(x) - u_{\eps_n}(z)|^p \, \dd x \, \dd z + \frac{1}{\eps_n} \int_{T_{Q^{(n)}_{j,s}}} V(u_{\eps_n}(x)) \, \dd x \r \\
& \,\, \leq \gamma_{\xi,n} \rho(x_j) G_{\eps_n}^{(p)}(u_{\eps_n},\rho(x_j),T_{Q_{j,s}^{(n)}})\,.
\end{align*}
Hence, using \eqref{eq:cond1}, \eqref{eq:condu0}, \eqref{eq:condun}, \eqref{eq:ratek} and~\eqref{eq:firstpartconmput}, we get
\begin{align*}
\cF_{\eps_n}^{(p)}(u_{\eps_n},A_i) & \leq \gamma_{\xi,n} \sum_{j=1}^{M_\xi} \sum_{s=1}^{k_j^{(n)}} \rho(x_j) G_{\eps_n}^{(p)}(u_{\eps_n},\rho(x_j),T_{Q^{(n)}_{j,s}}) \\
 & = \gamma_{\xi,n} \sum_{j=1}^{M_\xi} \sum_{s=1}^{k^{(n)}_j} \rho(x_j) \eps_n^{d-1} G_1^{(p)}(\widetilde{w}_j^\prime,\rho(x_j),T_{C_j}) \\
 & = \gamma_{\xi,n} \sum_{j=1}^{M_\xi} \frac{r_\xi^{d-1} \rho(x_j)}{\cH^{d-1}(C_j)} G_1^{(p)}(\widetilde{w}_j^\prime,\rho(x_j),T_{C_j}) \\
 & \leq \gamma_{\xi,n} \l r_\xi^{d-1} \sum_{j=1}^{M_\xi} \rho(x_j) \sigma^{(p)}(x_j,\nu) + 3\xi \r \\
 & \leq \gamma_{\xi,n} \l \int_{E_i} \sigma^{(p)}(x,\nu) \rho(x) \, \dd \cH^{d-1}(x) + 4\xi \r\,.
\end{align*}
Taking limsup as $n\to \infty$ and $\xi\to 0$ implies (ii).
Note that $u_{\eps_n}$ is not the recovery sequence, rather each $u_{\eps_n}$ depended on $\xi$ (through $L_j$ and $\beta_\xi$).
Hence, if we make the $\xi$ dependence explicit, we showed $\limsup_{n\to \infty} \cF_{\eps_n}^{(p)}(u_{\eps_n}^{(\xi)},A_i) \leq \limsup_{n\to \infty} \gamma_{\xi,n} \cG_\infty^{(p)}(u,A_i)$.
By a diagonalisation argument we can find a sequence $\xi_n\to 0$ such that $\limsup_{n\to \infty} \cF_{\eps_n}^{(p)}(u_{\eps_n}^{(\xi_n)},A_i) \leq \cG_\infty^{(p)}(u,A_i)$.
Moreover if the sequence $\xi_n\to 0$ sufficiently slowly then we can infer $\max L_j+\beta_{\xi_n} \geq \frac{1}{\zeta_{\eps_n}}$ and $\beta_{\xi_n}\leq \zeta_{\eps_n}$ hence, $\Lip(u_{\eps_n}) \leq \frac{1}{\zeta_{\eps_n} \eps_n}$ and $u_{\eps_n}(x)=u(x)$ whenever $\dist(x,\partial^*\{u=1\}) \geq \frac{\eps_n}{\zeta_{\eps_n}}$ for any sequence $\zeta_{\eps_n}\to 0$.

 \vspace{0.5\baselineskip}

We are thus left to prove claim (i).
Analogously to Remark~\ref{rem:nonlocldef}, $\widetilde{\Lambda}_{\eps_n}(u_{\eps_n},A_i,A_j)\to0$ as $n\to\infty$ for all $i,j$ for which $\mathrm{d}(A_i,A_j)>0$. 
Let us now consider indexes $i\neq j$ for which
$\mathrm{d}(A_i,A_j)=0$. Write
\begin{align*} 
\widetilde{\Lambda}_{\eps_n}(u_{\eps_n},A_i,A_j)
&=\frac{1}{\eps_n}\int_{B(0,R_\eta)}\int_{A_{\eps_n h}}
    \eta(h)|u_{\eps_n}(z+\eps_n h)-u_{\eps_n}(z)|^p \rho(z+\eps_n h)\rho(z)\,\dd z\,\dd h\,,
\end{align*}
where $A_{\eps_n h}:=\{ z\in A_i \,:\, z+\eps_n h\in A_j \}$.
Since $A_i$ and $A_j$ are Lipschitz graphs over sets lying on two different hyperplanes, it holds that $\Vol(A_{\eps_n h}) = O(\eps_n^d)$.
Thus
\[
\widetilde{\Lambda}_{\eps_n}(u_{\eps_n},A_i,A_j)\to0\,,
\]
as $n\to\infty$. \vspace{\baselineskip}

\emph{Step 2. The general case.} 
Let $u\in BV(X;\{\pm1\})$.
Using Theorem~\ref{thm:denspoly} it is possible to find a sequence $\{v_n\}_{n=1}^\infty$ of polyhedral function such that $v_n\to u$ in $L^1$ (which, in turn, implies that $Dv_n\stackrel{{w}^*}{\rightharpoonup}Du$) and $|D v_n|(X)\to|Du|(X)$.
Using Step 1 and a diagonalisation argument we get that there exists a sequence $\{u_n\}_{n=1}^\infty$ with $u_n\to u$ in $L^1(X)$ such that
\[
\cF^{(p)}_{\eps_n}(u_n)\leq \cG^{(p)}_\infty(v_n)+\frac{1}{n}\,.
\]
Then, Theorem~\ref{thm:rese} together with Lemma~\ref{lem:ConNLContinuum:Liminf:contsigma} gives us that
\[
\limsup_{n\to\infty}\cG^{(p)}_\infty(v_n)=\cG^{(p)}_\infty(u)\,.
\]
This concludes the proof.
\end{proof}


\section{Convergence of the Graphical Model \label{sec:ConvGraph}}

In this Section we prove Theorem~\ref{thm:MainRes:Compact&Gamma}.
In particular, in Section~\ref{subsec:ConvGraph:Compact} we prove the compactness part of Theorem~\ref{thm:MainRes:Compact&Gamma} and in Sections~\ref{subsec:ConvGraph:Liminf}-\ref{subsec:ConvGraph:Limsup} we prove the $\Gamma$-convergence result.


\subsection{Compactness \label{subsec:ConvGraph:Compact}}

\begin{proof}[Proof of Theorem~\ref{thm:MainRes:Compact&Gamma} (Compactness).]
\emph{Step 1.} We first show that there exist $c_n,\alpha_n>0$ with $\alpha_n,c_n \to 1$ as $n\to\infty$, such that
\begin{equation}\label{eq:esteta1}
\eta\l\frac{T_n(x)-T_n(z)}{\eps_n}\r \geq c_n \eta\l\frac{\alpha_n(x-z)}{\eps_n}\r\,.
\end{equation}
Let $\delta_n := \frac{2\|T_n-\Id\|_{L^\infty}}{\eps_n}$.
By Assumption~(C4) we can find $\alpha_n,c_n$ such that, for all $a,b\in \bbR^d$ with $|a-b|\leq \delta_n$, we have
\begin{equation}\label{eq:ineq1}
\eta(a) \geq c_n \eta(\alpha_n b)\,.
\end{equation}
Since by assumption (A2) we have that $\delta_n\to 0$ then $\alpha_n,c_n$ can be chosen such that $\alpha_n\to 1,c_n\to 1$.
Now if we let $a:=\frac{T_n(x)-T_n(z)}{\eps_n}$ and $b:=\frac{x-z}{\eps_n}$ we have
\[ |a-b| = \frac{|T_n(x) - T_n(z) + z - x|}{\eps_n} \leq \frac{2\|T_n-\Id\|_{L^\infty}}{\eps_n} = \delta_n \]
and therefore, by \eqref{eq:ineq1}, we get $\eta\l\frac{T_n(x)-T_n(z)}{\eps_n}\r \geq c_n \eta\l\frac{\alpha_n (x-z)}{\eps_n}\r$ as required. \vspace{0.5\baselineskip}

\emph{Step 2.} Let $v_n := u_n\circ T_n$. Using Lemma \ref{lem:writeint} and \eqref{eq:esteta1} we have that
\begin{align}
\cG_n^{(p)}(u_n) & = \frac{1}{\eps_n^{d+1}} \int_X\int_X \eta\l\frac{T_n(x)-T_n(z)}{\eps_n}\r |v_n(x) - v_n(z)|^p \rho(x) \rho(z) \, \dd x \, \dd z \nonumber \\
&\hspace{0.6cm}+ \frac{1}{\eps_n} \int_X V(v_n(x)) \rho(x) \, \dd x \nonumber \\
& \geq \frac{c_n}{\eps_n^{d+1}} \int_X\int_X \eta\l \frac{\alpha_n(x-z)}{\eps_n}\r |v_n(x) - v_n(z)|^p \rho(x) \rho(z) \, \dd x \, \dd z \nonumber \\
&\hspace{0.6cm}+ \frac{1}{\eps_n} \int_X V(v_n(x)) \rho(x) \, \dd x \nonumber \\
 & = \frac{c_n}{\alpha_n^{d+1}(\eps_n^\prime)^{d+1}} \int_X\int_X\eta\l \frac{x-z}{\eps_n^\prime}\r |v_n(x) - v_n(z)|^p \rho(x) \rho(z) \, \dd x \, \dd z \nonumber \\
&\hspace{0.6cm}+ \frac{1}{\eps_n} \int_X V(v_n(x)) \rho(x) \, \dd x \nonumber \\
 & = \frac{1}{\alpha_n} \cF_{\eps_n^\prime}^{(p)}(v_n) \label{eq:ineq2}
\end{align}
where $\eps_n^\prime := \frac{\eps_n}{\alpha_n}$ and $\cF_\eps^{(p)}$ is defined by~\eqref{eq:ConvNLContinuum:Feps} with $s_n := \frac{c_n}{\alpha_n^d}$ (in~\eqref{eq:ConvNLContinuum:Feps} $s$ depended on $\eps$ not on $n$, since we have fixed the sequence $\eps_n$ then clearly we could write $n$ in terms of $\eps_n^\prime$, however this would make the notation cumbersome). \vspace{0.5\baselineskip}

\emph{Step 3.} Since $\lim_{n\to\infty}\alpha_n=1$, we can infer that $\cF_{\eps_n^\prime}^{(p)}(v_n)$ is bounded and hence by Theorem~\ref{thm:ConvNLContinuum:ConvNLContinuum} the sequence $\{v_n\}_{n=1}^\infty$ is relatively compact in $L^1$. Therefore the sequence $\{u_n\}_{n=1}^\infty$ is relatively compact in $TL^1$, with any limit $u$ satisfying $\cG_\infty^{(p)}(u)<\infty$.
\end{proof}


\subsection{The Liminf Inequality \label{subsec:ConvGraph:Liminf}}

\begin{proof}[Proof of Theorem~\ref{thm:MainRes:Compact&Gamma} (Liminf).]
For any $u\in L^1(\mu)$ and any $u_n\in L^1(\mu_n)$ with $u_n\to u$ in $TL^1$ we claim that
\[ \liminf_{n\to \infty} \cG_n^{(p)}(u_n) \geq \cG_\infty^{(p)}(u). \]
Indeed, taking the liminf on both sides of \eqref{eq:ineq2}
and using Theorem~\ref{thm:ConvNLContinuum:ConvNLContinuum} we have
\[ \liminf_{n\to \infty} \cG_n^{(p)}(u_n) \geq \liminf_{n\to \infty} \frac{1}{\alpha_n} \cF_{\eps_n^\prime}^{(p)}(v_n) \geq \cG_\infty^{(p)}(u) \]
since $\lim_{n\to \infty} \frac{1}{\alpha_n} = 1$.
\end{proof}


\subsection{The Limsup Inequality \label{subsec:ConvGraph:Limsup} }

\begin{proof}[Proof of Theorem~\ref{thm:MainRes:Compact&Gamma} (Limsup).]
The aim of this section is to prove the following: given $u\in L^1(X,\mu)$ it is possible to find
a sequence $\{u_n\}_{n=1}^\infty\subset L^1(X_n)$ with $u_n\to u$ in $TL^1(X)$ such that 
\[
\limsup_{n\to \infty} \cG_n^{(p)}(u_n) \leq \cG_\infty^{(p)}(u)\,.
\]
Without loss of generality we can assume $\cG_\infty^{(p)}(u)<+\infty$. In particular, $u\in BV(X;\{\pm1\})$.

We divide the proof in two cases: we first assume that $u$ is a polyhedral function and then
we extend the argument to any function $u$ with $\cG_\infty^{(p)}(u)<+\infty$ via a diagonalisation argument. \vspace{0.5\baselineskip}

\emph{Case 1.} Assume that $u$ is a polyhedral function (see Definition \ref{def:polyfun}).
Let $w_n\in L^1(X,\mu)$ be a recovery sequence for the $\Gamma$-convergence of $\cF_{\eps_n^\prime}^{(p)}$ to $\cG_\infty^{(p)}$ where $\eps_n^\prime = \frac{\eps_n}{\hat{\alpha}_n}$ and $\hat{\alpha}_n$ is a sequence we fix shortly (see step 2)
\emph{i.e.}, we assume $w_n\to u$ in $L^1(X,\mu)$ and $\lim_{n\to\infty} \cF_{\eps_n^\prime}^{(p)}(w_n) = \cG_\infty^{(p)}(u)$.
Let $u_n(x_i) = n\int_{T_n^{-1}(x_i)} w_n(x) \, \dd x$ where $T_n$ is any sequence of transport maps satisfying the conclusions of Theorem~\ref{thm:optrate}.
\vspace{0.5\baselineskip}

\emph{Step 1.} We show that $u_n\to u$ in $TL^1$.
Setting $v_n = u_n\circ T_n$, we are required to show $v_n\to u$ in $L^1$.
Let $\zeta_n\to 0$ with $\zeta_n\gg \sqrt{\frac{\|T_n-\Id\|_{L^\infty}}{\eps_n}}$. 
By Theorem~\ref{thm:ConvNLContinuum:ConvNLContinuum} we can assume that $w_n$ are Lipschitz continuous with $\Lip(w_n)\leq \frac{1}{\zeta_n\eps_n}$ and $w_n(x) = u(x)$ for all $x$ satisfying $\dist(x,\partial^*\{u=1\})>\frac{\eps_n}{\zeta_n}$.
Hence, $v_n(x) = w_n(x) = u(x)$, for $n$ sufficiently large, for all $x$ such that $\dist(x,\partial^*\{u=1\})>\frac{3\eps_n}{\zeta_n}$. 
Therefore,
\begin{align*}
\| v_n - w_n \|_{L^1(X)} & = \sum_{i=1}^n \int_{T_n^{-1}(x_i)} |v_n(x) - w_n(x) | \, \dd x \\
 & \leq n \sum_{i=1}^n \int_{T_n^{-1}(x_i)} \int_{T_n^{-1}(x_i)} |w_n(y) - w_n(x) | \, \dd y \, \dd x \\
 & \leq \frac{2 \|T_n-\Id\|_{L^\infty}}{\zeta_n \eps_n n} \#\lb i\,:\, \dist(x_i,\partial^*\{u=1\})\leq \frac{3\eps_n}{\zeta_n} \rb\,.
\end{align*}
Now,
\begin{align*}
\#\lb i\,:\, \dist(x_i,\partial^*\{u=1\})\leq \frac{3\eps_n}{\zeta_n} \rb & = n\mu\l \lb x\,:\, \dist(T_n(x),\partial^*\{u=1\})\leq \frac{3\eps_n}{\zeta_n} \rb\r \\
 & \leq nc_2 \Vol\l\lb x\,:\, \dist(x,\partial^*\{u=1\})\leq \frac{4\eps_n}{\zeta_n} \rb\r \\
 & = O\l\frac{n\eps_n}{\zeta_n}\r.
\end{align*}
Hence, 
\begin{equation} \label{eq:ConvGraph:Limsup:vnwn}
\|v_n-w_n\|_{L^1(X)} = o(\eps_n)\,.
\end{equation}
Moreover,
\[ \|v_n-u\|_{L^1(X)} \leq \|v_n-w_n\|_{L^1(X)} + \|w_n-u\|_{L^1(X)} \to 0\,, \]
so $u_n\to u$ in $TL^1$.\vspace{0.5\baselineskip}

\emph{Step 2.}  We show that there exists $\hat{\alpha}_n,\hat{c}_n\to 1$ such that
\begin{equation} \label{eq:ConvGraph:Limsup:EtaUB}
\eta\l \frac{T_n(x) - T_n(z)}{\eps_n} \r \leq \hat{c}_n \eta\l\frac{\hat{\alpha}_n (x-z)}{\eps_n}\r.
\end{equation}
To show the validity of \eqref{eq:ConvGraph:Limsup:EtaUB} we use the following subclaim: for all
$\hat{\delta}>0$ sufficiently small there exists $\hat{\alpha}_{\hat{\delta}},\hat{c}_{\hat{\delta}}>0$ such that $\hat{\alpha}_{\hat{\delta}}\to 1, \hat{c}_{\hat{\delta}}\to 1$, as $\hat{\delta}\to 0$, and, for any $\hat{a},\hat{b}\in\bbR^d$, it holds
\begin{equation}\label{eq:claim1}
|\hat{a}-\hat{b}|<\hat{\delta}\quad\Rightarrow\quad\hat{c}_{\hat{\delta}} \eta(\hat{\alpha}_{\hat{\delta}} \hat{b}) \geq \eta(\hat{a})\,.
\end{equation}
Then~\eqref{eq:ConvGraph:Limsup:EtaUB} can be obtained as follows: for any $n\in\bbN$ take
\[
\hat{a} := \frac{T_n(x) - T_n(z)}{\eps_n}\,,\quad\quad
\hat{b} := \frac{x-z}{\eps_n}\,,\quad\quad
\hat{\delta} := \frac{2\|T_n-\Id\|_{L^\infty}}{\eps_n}\,,
\]
and let $\hat{\alpha}_n,\hat{c}_n$ be the numbers given by the subclaim for which \eqref{eq:claim1} holds.
Note that $\hat{\alpha}_n$ is chosen independently from $w_n$ (since $w_n$ depends on $\hat{\alpha}_n$ there is therefore no circular argument).
Then, since $|\hat{a}-\hat{b}|\leq \delta_n$, we infer~\eqref{eq:ConvGraph:Limsup:EtaUB}.

To prove the subclaim, we let $\alpha_\delta,c_\delta$ be as in Assumption~(C4) and let $\hat{\delta} >0$.
Without loss of generality we assume that $\inf_{\gamma\in (0,1]}\alpha_\gamma\in (0,\infty)$.
We choose $\delta := \min\lb 1,\frac{\hat{\delta}}{\inf_{s\in (0,1]} \alpha_s}\rb$, trivially $\delta\to 0$ as $\hat{\delta}\to 0$.
We assume that $\frac{\hat{\delta}}{\inf_{\gamma\in (0,1]} \alpha_\gamma} \leq 1$.
Let $\hat{a},\hat{b}\in \bbR^d$ with $|\hat{a}-\hat{b}|<\hat{\delta}$, and define $a := \frac{\hat{a}}{\alpha_\delta}$ and $b := \frac{\hat{b}}{\alpha_\delta}$.
Since, $|a-b|\leq \frac{\hat{\delta}}{\alpha_\delta} \leq \frac{\hat{\delta}}{\inf_{\gamma\in (0,1]}\alpha_\gamma} = \delta$ then
\[
\eta(b) \geq c_\delta \eta(\alpha_\delta a) \quad \Rightarrow \quad \frac{1}{c_\delta} \eta\l \frac{\hat{b}}{\alpha_\delta} \r \geq \eta(\hat{a})\,.
\]
Let $\hat{c}_{\hat{\delta}} := \frac{1}{c_\delta}$, $\hat{\alpha}_{\hat{\delta}} := 1/\alpha_\delta$ then $\hat{\delta}\to 0$ implies $\delta\to 0$ which in turn implies $\alpha_\delta,c_\delta\to 1$ and therefore $\hat{\alpha}_{\hat{\delta}},\hat{c}_{\hat{\delta}}\to 1$.
This proves the claim. \vspace{0.5\baselineskip}

\emph{Step 3.} 
Using Lemma \ref{lem:writeint} and \eqref{eq:ConvGraph:Limsup:EtaUB} we get
\begin{align*}
\cG^{(p)}_n(u_n) & = \frac{1}{\eps_n} \int_X\int_X \eta_{\eps_n}(T_n(x) - T_n(z)) |v_n(x) - v_n(z)|^p \rho(x) \rho(z) \, \dd x \, \dd z + \frac{1}{\eps_n} \int_X V(v_n(x)) \rho(x) \, \dd x \\
 & \leq \frac{\hat{c}_n}{\hat{\alpha}_n^{d+1}\eps^\prime_n} \int_X\int_X \eta_{\eps^\prime_n}( x-z) |v_n(x) - v_n(z) |^p \rho(x) \rho(z) \, \dd x \, \dd z + \frac{1}{\eps_n} \int_X V(v_n(x)) \rho(x) \, \dd x \\
 & = \frac{\hat{c}_n}{\hat{\alpha}_n^{d+1}\eps^\prime_n} \int_X\int_X \eta_{\eps^\prime_n}( x-z) |w_n(x) - w_n(z) |^p \rho(x) \rho(z) \, \dd x \, \dd z + \frac{1}{\eps_n} \int_X V(w_n(x)) \rho(x) \, \dd x \\
 & \hspace{1cm} + a_n + b_n
\end{align*}
where we recall $\eps_n^\prime := \frac{\eps_n}{\hat{\alpha}_n}$ and
\begin{align*}
a_n & := \frac{\hat{c}_n}{\hat{\alpha}_n^{d+1} \eps_n^\prime} \int_X\int_X \eta_{\eps_n^\prime}(x-z)\l |v_n(x) - v_n(z)|^p - |w_n(x) - w_n(z)|^p \r \rho(x) \rho(z) \, \dd x \, \dd z\,, \\
b_n & := \frac{1}{\eps_n} \int_X \l V(v_n(x)) - V(w_n(x)) \r \rho(x) \, \dd x\,. 
\end{align*}
We recall the followings inequalities: $\forall \delta>0$ there exits $C_\delta>0$ such that for any $a,b\in \bbR^d$ we have
\begin{equation}\label{eq:inequality1}
|a|^p \leq (1+\delta)|b|^p + C_\delta |a-b|^p\,.
\end{equation}
Inequality \eqref{eq:inequality1} follows by noticing that, for every given $\delta>0$, the function $x\mapsto\frac{|x|^p-1-\delta}{|x-1|^p}$ is bounded in $[0,\infty)$, and then setting $x=\frac{|a|}{|b|}$.
Moreover, for all $p\geq1$ and all $a,b\in\bbR$, it holds
\begin{equation}\label{eq:inequality2}
|a+b|^p\leq 2^{p-1}\l |a|^p+|b|^p \r\,.
\end{equation}
Fix $\delta>0$. Using \eqref{eq:inequality1} and \eqref{eq:inequality2} we infer
\[
a_n \leq \frac{\hat{c}_n \delta}{\hat{\alpha}_n^{d+1} \eps_n^\prime} \int_X\int_X \eta_{\eps_n^\prime}(x-z) |w_n(x) - w_n(z)|^p \rho(x) \rho(z) \, \dd x \, \dd z + d_n
\]
where
\[
d_n := \frac{2^{p-1}\hat{c}_n C_\delta}{\hat{\alpha}_n^{d+1} \eps_n^\prime} \int_X\int_X \eta_{\eps_n^\prime}(x-z) \Big( |v_n(x) - w_n(x)|^p + |v_n(z) - w_n(z)|^p \Big) \rho(x) \rho(z) \, \dd x \, \dd z\,.
\]
We show $d_n \to 0$. We have that
\begin{align*}
d_n & \leq  \frac{2^{p} \hat{c}_nC_\delta c_2^2}{\hat{\alpha}_n^{d+1} \eps_n^\prime} \int_X |v_n(x) - w_n(x)|^p \, \dd x \int_{\bbR^d} \eta(x) \, \dd x \\
 & \leq \frac{2^{2p-1} \hat{c}_nC_\delta c_2^2}{\hat{\alpha}_n^{d+1} \eps_n^\prime} \| v_n-w_n\|_{L^1(X)} \int_{\bbR^d} \eta(x) \, \dd x \\
 & = O\l\frac{C_\delta \|v_n-w_n\|_{L^1(X)}}{\eps_n}\r\,,
\end{align*}
where we used the fact that $\hat{c}_n\to1$, $\hat{\alpha}_n\to 1$ as $n\to\infty$.
From Equation~\eqref{eq:ConvGraph:Limsup:vnwn} we get that
\begin{equation}\label{eq:bn}
d_n\to0\,,
\end{equation}
as $n\to\infty$.

Let $L_V$ be the Lipschitz constant for $V$ (as given by Assumption~(B4)).
Since,
\begin{align*}
b_n & \leq \frac{c_2}{\eps_n} \int_X \la V(v_n(x)) - V(w_n(x)) \ra \, \dd x \leq \frac{c_2 L_V}{\eps_n} \|v_n-w_n\|_{L^1(X)}
\end{align*}
then $b_n\to 0$ by~\eqref{eq:ConvGraph:Limsup:vnwn}.

Combining the above estimates we have,
\begin{equation}\label{eq:kn}
\begin{split}
\cG^{(p)}_n(u_n) & \leq \frac{\hat{c}_n(1+\delta)}{\hat{\alpha}_n^{d+1} \eps_n^\prime} \int_X \int_X \eta_{\eps_n^\prime}(x-z) |w_n(x) - w_n(z)|^p \rho(x) \rho(z) \, \dd x \, \dd z \\
 & \hspace{1cm} + \frac{1}{\eps_n} \int_X V(w_n(x)) \rho(x) \, \dd x + b_n + d_n\,.
\end{split}
\end{equation}
\vspace{0.5\baselineskip}

\emph{Step 4.} We now conclude as follows.
Using \eqref{eq:kn} we get
\[ \cG_n^{(p)}(u_n) \leq \frac{(1+\delta)}{\hat{\alpha}_n} \cF_{\eps_n^\prime}^{(p)}(w_n) + b_n + d_n \]
for $\cF_{\eps_n^\prime}^{(p)}$ defined as in~\eqref{eq:ConvNLContinuum:Feps} with $s_n := \frac{\hat{c}_n}{\hat{\alpha}_n^d}$.
Taking the limsup on both sides and using \eqref{eq:bn}, together with Theorem~\ref{thm:ConvNLContinuum:ConvNLContinuum} we have
\[ \limsup_{n\to \infty} \cG_n^{(p)}(u_n) \leq \limsup_{n\to \infty} (1+\delta) \cF_{\eps_n^\prime}(w_n) \leq (1+\delta) \cG_\infty^{(p)}(u). \]
Taking $\delta\to 0$ completes the proof for case 1.
\vspace{0.5\baselineskip}

\emph{Case 2.} To extend to arbitrary functions $u\in BV(X;\{\pm1\})$
we apply the following diagonalisation argument.
Using Theorem \ref{thm:denspoly} together with Lemma \ref{lem:ConNLContinuum:Liminf:contsigma} and Theorem \ref{thm:rese} we can find a sequence of polyhedral functions $\{u^{(m)}\}_{m=1}^\infty$ such that
\[
\|u-u^{(m)}\|_{L^1}\leq \frac{1}{m}\,,\quad\quad
\cG_\infty^{(p)}(u^{(m)}) \leq \cG_\infty^{(p)}(u) + \frac{1}{m} \,.
\]
Using the result of Case 1, for each $m\in \bbN$ we have
\[
\limsup_{n\to \infty} \cG_n^{(p)}(u_n^{(m)}) \leq \cG_\infty^{(p)}(u^{(m)})\,,
\]
where $u_n^{(m)}$ is the recovery sequence for $u^{(m)}$ in the $\Gamma$-convergence $\cG_\infty^{(p)} = \Glim_{n\to \infty} \cF_{\eps_n^\prime}^{(p)}$.
For each $m\in\bbN$ let $n_m\in\bbN$ be such that
\[
\cG_n^{(p)}(u_n^{(m)}) \leq \cG_\infty^{(p)}(u^{(m)}) + \frac{1}{m} \quad\text{ and }\quad
\| u_n^{(m)}\circ T_n - u^{(m)}\|_{L^1} \leq \frac{1}{m}
\]
for all $n\geq n_m$.
At the cost of increasing $n_m$ we assume that $n_{m+1}> n_m$ for all $m$.
Let $u_n:=u_n^{(m)}$ for $n\in [n_m,n_{m+1})$.
Then,
\[ \limsup_{n\to \infty} \cG_n^{(p)}(u_n) = \limsup_{m\to \infty} \sup_{n\in [n_m,n_{m+1})} \cG_n^{(p)}(u_n^{(m)}) \leq \limsup_{m\to \infty} \l \cG_\infty^{(p)}(u^{(m)}) + \frac{1}{m} \r \leq \cG_\infty^{(p)}(u)\,. \]
Similarly,
\[ \lim_{n\to \infty} \|u_n\circ T_n - u\|_{L^1} \leq \limsup_{m\to \infty} \sup_{n\in [n_m,n_{m+1})} \l \| u_n^{(m)}\circ T_n - u^{(m)} \|_{L^1} + \frac{1}{m} \r \leq \lim_{m\to \infty} \frac{2}{m} = 0 \]
therefore $u_n$ converges to $u$ in $TL^1$.
Hence $u_n$ is a recovery sequence for $u$.
This completes the proof.
\end{proof}


\section{Convergence of Minimizers with Data Fidelity \label{sec:Const}}

In this section we prove Corollary \ref{cor:MainRes:Constrained}.

\begin{proof}[Proof of Corollary~\ref{cor:MainRes:Constrained}.]
In view of Theorem \ref{thm:convmin} and Proposition \ref{prop:contconv2} it is enough to prove that, for any $u\in L^1(X,\mu)$ and any $\{u_n\}_{n=1}^\infty$ with $u_n\in L^1(X_n)$ such that $u_n\to u$ in $TL^1(X)$ it holds that
\[
\lim_{n\to \infty} \cK_n(u_n) = \cK_\infty(u)\,.
\]
We can restrict ourselves to sequences $u_n\to u$ satisfying
\begin{equation}\label{eq:hypGn}
\sup_{n\in \bbN} \cG_n^{(p)}(u_n)< +\infty\,.
\end{equation}
Let
\[
v_n(x) :=
\left\{
\begin{array}{ll}
u_n(x) & \text{if } x\in X_n(u_n)\,,\\
1 &\text{otherwise}\,.
\end{array}
\right.
\]
where
\[
X_n(u_n) := \lb x\in X_n \, : \, |u_n(x)|^q\leq R_V \rb\,,
\]
and $R_V$ is as in Assumption~(B3). \vspace{0.5\baselineskip}

\emph{Step 1.} We claim that
\begin{equation}\label{eq:convDnunvn}
\lim_{n\to \infty} |\cK_n(u_n) - \cK_n(v_n) | \to 0\,,
\end{equation}
as $n\to\infty$.
Indeed we have that
\begin{align*}
\lim_{n\to \infty} \la \cK_n(u_n) - \cK_n(v_n) \ra & = \lim_{n\to \infty} \la \frac{1}{n} \sum_{x_i\not\in X_n(u_n)} \l k_n(x_i,u_n(x_i)) - k_n(x_i,1) \r \ra \\
 & = \beta \lim_{n\to \infty} \frac{1}{n} \sum_{x_i\not\in X_n(u_n)} \l 3+|u_n(x_i)|^q \r,
\end{align*}
where in the last equality we used (D2).
Now,
\[ \frac{1}{n} \sum_{x_i\not\in X_n(u_n)} 1 \leq \frac{1}{nR_V} \sum_{x_i\not\in X_n(u_n)} |u_n(x_i)|^q \leq \frac{1}{nR_V \tau} \sum_{x_i\not\in X_n(u_n)} V(u_n(x_i)) \leq \frac{\eps_n \cG_n^{(p)}(u_n)}{\tau R_V}. \]
Using \eqref{eq:hypGn} we conclude~\eqref{eq:convDnunvn}. \vspace{0.5\baselineskip}

\emph{Step 2.} We claim that $v_n\to u$ in $TL^1(X)$.
By a direct computation we get
\begin{align*}
\| u_n - v_n\|_{L^1(\mu_n)} & = \frac{1}{n} \sum_{x_i \not\in X_n(u_n)} |u_n(x_i) - 1| \\
 & \leq \frac{1}{n} \sum_{x_i\not\in X_n(u_n)} \l 1 + |u_n(x_i)| \r \\
 & \leq \frac{1}{n} \sum_{x_i\not\in X_n(u_n)} \l 1 + |u_n(x_i)|^q \r \to 0,
\end{align*}
so $v_n\to u$ in $TL^1$. \vspace{0.5\baselineskip}

\emph{Step 3.} We show that $\cK_n(v_n) \to \cK_\infty(u)$.
Consider any subsequence of $v_n$ which we do not relabel.
From Step 2 we have that $v_n\circ T_n\to u$ in $L^1$. Thus there exists a further subsequence (not relabeled) such that $v_n(T_n(x)) \to u(x)$ for almost every $x\in X$.
Using Assumption (D3) we get
\begin{equation}\label{eq:domconv1}
\lim_{n\to \infty} k_n\l T_n(x),v_n(T_n(x)) \r = k_\infty(x,u(x))\,.
\end{equation}
Moreover it holds
\begin{equation}\label{eq:domconv2}
k_n(T_n(x),v_n(T_n(x))) \leq \beta\l 1+|v_n(T_n(x))|^q \r \leq \beta\l 1 + R_V \r
\end{equation}
Using \eqref{eq:domconv1}, \eqref{eq:domconv2} and applying the Lebesgue's dominated convergence theorem we get
\begin{align}\label{eq:convDn}
\lim_{n\to \infty} \cK_n(v_n) & = \lim_{n\to \infty} \int_X k_n(T_n(x),v_n(T_n(x))) \, \rho(x) \, \dd x \nonumber \\
 & = \int_X k_\infty(x,u(x)) \rho(x) \, \dd x \nonumber\\
 & = \cK_\infty(u)
\end{align}
Since any subsequence of $v_n$ has a further subsequence such that $\lim_{n\to \infty} \cK_n(v_n) = \cK_\infty(u)$ then we can conclude the convergence is over the full sequence. \vspace{0.5\baselineskip}

\emph{Step 4.} Using \eqref{eq:convDnunvn} and \eqref{eq:convDn} we conclude that
\[
\lim_{n\to\infty}\cK_n(u_n)\to \cK_\infty(u)
\]
as required
\end{proof}


\section*{Acknowledgements}
The authors thank the Center for Nonlinear Analysis at Carnegie Mellon University and the Cantab Capital Institute for the Mathematics of Information at the University of Cambridge for their support during the preparation of the manuscript.
In addition the authors would like to thank Dejan Slep\v{c}ev and Nicol\'{a}s Garc\'{i}a Trillos for discussions and references that improved the manuscript.
The matlab implementation of the bean used in the example in Section~\ref{subsec:Intro:pEx} was written by Dejan Slep\v{c}ev.
The example in Section~\ref{subsec:Intro:AnisoEx} was aided by Luca Calatroni.
The research of R.C. was funded by National Science Foundation under Grant No. DMS-1411646.
Part of the research of M.T. was funded by the National Science Foundation under Grant No. CCT-1421502.
Both authors would like to thank the referees for valuable feedback that improved the paper.


\bibliographystyle{siam}
\bibliography{ref_GL2}

\end{document}